\documentclass{elsarticle}

\usepackage{geometry}
\usepackage{color}
\usepackage{bm}
\usepackage{multirow}
\usepackage{hyperref}
\usepackage{amsopn}

\newcommand{\supp}{\mathop{\mathrm{supp}}}

\usepackage{amsfonts}
\usepackage{amsmath}
\usepackage{amsthm}
\usepackage{graphicx}
\usepackage{epstopdf}

\usepackage{algorithmic}
\usepackage{algorithm}
\ifpdf
  \DeclareGraphicsExtensions{.eps,.pdf,.png,.jpg}
\else
  \DeclareGraphicsExtensions{.eps}
\fi

\newtheorem{remark}{Remark}[section]
\newtheorem{definition}{Definition}[section]
\newtheorem{theorem}{Theorem}[section]

\journal{Journal of Computational Physics}

\begin{document}

\begin{frontmatter}

\title{Using hierarchical matrices in the solution of the time-fractional heat equation by multigrid waveform relaxation\tnoteref{mytitlenote}}
\tnotetext[mytitlenote]{}





\author[Tufts]{Xiaozhe Hu}\ead{xiaozhe.hu@tufts.edu}
\address[Tufts]{Department of Mathematics, Tufts University, 
Medford, Massachusetts  02155, USA}

\author[UniZar]{Carmen Rodrigo\corref{cor1}}\ead{carmenr@unizar.es}
\address[UniZar]{IUMA and Applied Mathematics Department, University of Zaragoza, Zaragoza, Spain}

\author[CWI]{Francisco J. Gaspar}\ead{F.J.Gaspar@cwi.nl}
\address[CWI]{CWI, Centrum Wiskunde \& Informatica, Science Park 123, 1090 Amsterdam, The Netherlands}

\cortext[cor1]{Corresponding author. Tel.: +34 976762148; E-mail address: carmenr@unizar.es (C. Rodrigo)}

\begin{abstract}
This work deals with the efficient numerical solution of the time-fractional heat equation discretized on non-uniform temporal meshes.  Non-uniform grids are essential to capture the singularities of ``typical'' solutions of time-fractional problems.  We propose an efficient space-time multigrid method based on the waveform relaxation technique, which accounts for the nonlocal character of the fractional differential operator. 
To maintain an optimal complexity, which can be obtained for the case of uniform grids, we approximate the coefficient matrix corresponding to the temporal discretization by its hierarchical matrix (${\cal H}$-matrix) representation.  In particular, the proposed method has a computational cost of ${\cal O}(k N M \log(M))$, where $M$ is the number of time steps, $N$ is the number of spatial grid points, and $k$ is a parameter which controls the accuracy of the ${\cal H}$-matrix approximation.  The efficiency and the good convergence of the algorithm, which can be theoretically justified by a semi-algebraic mode analysis, are demonstrated through numerical experiments in both one- and two-dimensional spaces. 
\end{abstract}


\end{frontmatter}

%\linenumbers

%%%%%%%%%%%%%%%%%%%%%%%%%%%%%%%%%%%%%%%%%%%%%%%%%

\section{Introduction}\label{sec:1}

The design of efficient numerical methods for differential equations involving fractional derivatives has become a challenging topic recently, due to its wide range of applications in many different fields~\cite{cushman, hall, hilfer2, lazarov,li_xu, 11hesthaven,metzler_klafter, purohit, Ren.J;Sun.Z;Zhao.X2013a,10hesthaven,Zhao.X;Sun.Z2011a}.   The nonlocal character of the fractional differential operator usually results in a dense coefficient matrix when numerical methods are applied, leading to higher memory requirements as well as a considerable increase of the solution time comparing with solving integer differential equations.  Usually, if uniform meshes are employed, the coefficient matrix has a Toeplitz-like structure, and efficient solvers have been proposed to reduce the computational complexity of traditional Gaussian elimination type methods.  
%This happens when a uniform mesh is employed and then, efficient computations can be performed thanks to such Toeplitz structure. 
For time-fractional differential equations, alternating direction implicit schemes (ADI) with a computational complexity of ${\cal O}(NM^2)$, where $N$ is the number of spatial grid-points and $M$ the number of time steps, were proposed in \cite{Zhang2014}.  Also an approximate inversion method with a computational cost of ${\cal O}(N M \log(M))$ has been proposed in \cite{Lin2016}, where the authors approximate the coefficient matrix by a block $\varepsilon$-circulant matrix, which can be block diagonalized by FFT and,  in order to solve the resulting complex block system, the authors use a multigrid method.  A parallel-in-time method based on the parareal algorithm \cite{LionsMaday} has been proposed for solving time-fractional differential equations \cite{Xu_Hesthaven_Chen}. This method consists of an iterative predictor-corrector procedure combining an inexpensive but inaccurate solver with an expensive but accurate solver.
Also, an efficient algorithm for the evaluation of the Caputo fractional derivative, based on the sum-of-exponentials technique, is presented in \cite{fast_Caputo} for the solution of fractional diffusion equations.
Recently, an efficient, robust and parallel-in-time multigrid method based on the waveform relaxation approach has been proposed in \cite{SISC_Gaspar}. By exploiting the Toeplitz-like structure of the coefficient matrix, the computational complexity of the method is ${\cal O}(N M \log(M))$ with a storage requirement of ${\cal O}(NM)$.   
Although some works have studied the solution of fractional diffusion equations on non-uniform temporal meshes, see for example \cite{nonuniform}, most existing efficient solvers can only be applied when a uniform grid is used, or they can only achieve their best performance when that is the case.
%To the best of our knowledge, however, most existing efficient solvers can only be applied when a uniform grid is used, or they can only achieve their best performance when that is the case.  
The latter happens for the fast solver presented in \cite{SISC_Gaspar}, which can be applied for nonuniform meshes but with a significant increase of the computational complexity, due to the loss of the Toeplitz structure of the matrix.  In the case of space-fractional PDEs, recently, a fast solver based on a geometric multigrid method for nonuniform grids has been proposed  in~\cite{fractional_xiaozhe}. The key in this work is to use hierarchical matrices to approximate the dense stiffness matrices.  Our aim here is to combine the two approaches proposed in \cite{SISC_Gaspar, fractional_xiaozhe} and efficiently solve, both in terms of CPU time and memory requirements, the time-fractional heat equation on nonuniform grids.
  
${\cal H}$-matrices \cite{Bebendorf.M2008a, H_matrices_Hackbusch_book} consist of powerful data-sparse approximations of dense matrices, providing a significant reduction of the storage requirement from ${\cal O}(n^2)$ to ${\cal O}(nk\log(n))$ units of storage ($n$ is the matrix size), where $k$ is a parameter that controls the accuracy of the approximation.  Moreover, the matrix-vector multiplication in ${\cal H}$-matrix format can be done in ${\cal O}(kn\log(n))$ operations.  Therefore, on non-uniform grids, the approximation of the dense matrices arising from the time discretization of the fractional partial differential equations (PDEs) by the ${\cal H}$-matrices representation is the key for maintaining an optimal computational complexity of the multigrid waveform relaxation method~\cite{horton_vandewalle, lubich_ostermann, Vandewalle_book}. 

As in standard multigrid methods (see~\cite{Stu_Tro, TOS01, Wess}), the multigrid waveform relaxation method accelerates the convergence of the waveform relaxation method, which is a continuous-in-time iterative algorithm for solving large systems of ordinary differential equations (ODEs), by introducing a hierarchy of coarser levels.  This method combines the very fast multigrid convergence with the high parallel efficiency of the waveform relaxation.  In practice, it uses a red-black zebra-in-time line relaxation together with a coarse-grid correction procedure based on coarsening only in the spatial dimension.  In this work, we develop an efficient and robust multigrid waveform relaxation method based on the ${\cal H}$-matrix representation of the discretization of the time-fractional heat equation on a nonuniform temporal grid.  The good convergence properties of the algorithm will be theoretically justified by applying a semi-algebraic mode analysis (SAMA)~\cite{sama}.  This analysis is essentially a generalization of the classical local Fourier analysis (LFA) or local mode analysis~\cite{Bra77, Bra94, TOS01, Wess, Wie01} and combines the standard LFA with an algebraic computation that accounts for the non-local character of the fractional differential operators. 
\textcolor{red}{In particular, the semi-algebraic analysis developed in \cite{SISC_Gaspar} for the time-fractional heat equation is used in this work.}
In addition, the classical disadvantage of waveform methods regarding the requirement of extra storage for unknowns is not a drawback here anymore since the storage of the solutions in previous steps is needed anyway for the time-fractional PDEs.

The remainder of this work is structured as follows. In Section~\ref{sec:2}, we introduce the time-fractional model problem and its discretization in the general framework of a non-uniform temporal grid.  Section~\ref{sec:3} is devoted to present the hierarchical matrix representation of the dense matrix corresponding to the time-discretization.  In Section~\ref{sec:4}, the multigrid waveform relaxation method is described, and its computational cost is estimated based on the computations in ${\cal H}$-matrix framework.  In order to illustrate the good behavior of the multigrid waveform relaxation method for solving the time-fractional diffusion problem, in Section~\ref{sec:5}, numerical experiments in both one- and two-dimensional spaces are considered.  In addition, some results of the semi-algebraic mode analysis are also presented in this section to theoretically confirm the good convergence results obtained numerically.  Finally, some conclusions are drawn in Section~\ref{sec:6}.

%%%%%%%%%%%%%%%%%%%%%%%%%%%%%%%%%%%%%%%%%%%%%%%%%

\section{Model problem and discretization}\label{sec:2}

We consider the time-fractional heat equation, arising by replacing the first-order time derivative with the Caputo derivative of order $\delta$, where $0<\delta<1$. In the literature, this model is also known as fractional sub-diffusion equation, which is a subclass of anomalous diffusive problems \cite{Hilfer, Podlubny}. In this section, we restrict ourselves to the one-dimensional case for the sake of simplicity and formulate the model problem as the following initial-boundary value problem,
\begin{align}
&D_t^{\delta} u - \frac{\partial^2 u}{\partial x^2} = f(x,t),\quad 0<x<L,\; t>0,\label{model_IVP_1}\\
&u(0,t) = 0,\; u(L,t)=0,\quad t>0,\label{model_IVP_2}\\
&u(x,0) = g(x), \quad 0\leq x\leq L.\label{model_IVP_3}
\end{align}
Here $D_t^{\delta}$ denotes the Caputo fractional derivative~\cite{Diethelm,stynes_graded}, defined as follows
\begin{equation*} %\label{caputo}
D_t^{\delta} u (x,t):= \left[J^{1-\delta}\left(\frac{\partial u}{\partial t}\right)\right](x,t),\quad 0\leq x\leq L,\; t>0,
\end{equation*}
where $J^{1-\delta}$ is the Riemann-Liouville fractional integral operator given by
\begin{equation*} %\label{RiemanLiouville}
\left(J^{1-\delta}u\right)(x,t):=\left[\frac{1}{\Gamma(1-\delta)}\int_{0}^{t} (t-s)^{-\delta}u(x,s)ds\right],\quad 0\leq x\leq L,\; t>0,
\end{equation*}
with $\Gamma$ being the Gamma function~\cite{gamma}. 

In order to discretize model problem~\eqref{model_IVP_1}-\eqref{model_IVP_3} we consider a uniform mesh in space 
\begin{equation*} %\label{spatial_mesh}
G_{h} = \left\{x_n = nh,\, n=0,1,\ldots,N+1\right\},
\end{equation*}
where $h=\displaystyle \frac{L}{N+1}$ and $N+1$ is the number of subdivisions in the spatial domain, and a non-uniform grid in time, $G_{\tau}$ given by $0 = t_0<t_1<\cdots <t_{M-1}<t_M = T$ with $T$ being the final time, $M$ representing the number of subdivisions for temporal discretization and the time step size is $\tau_m = t_{m+1}-t_m, \ m=1,\ldots,M-1$. 
Then, the whole grid is given by $G_{h,\tau} = G_h \times G_{\tau}$. For the sake of simplicity, we use a uniform spatial grid in the presentation. Notice, however, that the proposed method can be straightforwardly applied on non-uniform grids.

The diffusion term in \eqref{model_IVP_1} is approximated by standard spatial discretization schemes such as finite difference or finite element methods.  Here, we use a standard second order finite difference scheme, yielding the following semi-discrete problem
\begin{equation}\label{system_ODEs}
D_t^{\delta} u_h(t) + A_hu_h(t) = f_h(t),\; u_h(0) = g_h,\; t>0,
\end{equation}
where $u_h$ and $f_h$ are functions at time $t$ defined on the discrete spatial mesh $G_h$, and $A_h$ is the discrete space approximation.
For the time discretization, we use a Petrov-Galerkin approach.  In order to describe such an approximation, we consider the following simplified problem,
\begin{equation}\label{simple_eq}
D_t^{\delta} u(t) = f(t), \  t\in(0,T], 
\end{equation}
with initial condition $u(0) = 0$.
As standard in the Galerkin finite element framework, we consider the finite dimensional space ${\mathcal V} :=\hbox{span}\{\varphi_1, \ldots, \varphi_{M}\}$, where $\varphi_i$ are the standard piecewise linear basis functions defined on $G_{\tau}$. For the test functions we consider the following Dirac's delta functions $\psi_m(t) = \delta(t-t_m)$. Next, equation \eqref{simple_eq} is multiplied by $\psi_m(t)$ and integrated over interval $(0,T)$, which gives
\begin{equation*}
\int_0^T D_t^{\delta} u(t) \ \psi_m(t) {\rm d}t = \int_0^T f(t) \psi_m(t) {\rm d}t.
\end{equation*}
Since $u(t) \approx u^M(t) = \sum_{j=1}^M u_j\varphi_j(t)$, by using the definition of the Caputo fractional derivative, we obtain that
\begin{equation*}
\sum_{j=1}^M \left( \frac{1}{\Gamma(1-\delta)}\int_0^T \left( \int_0^t (t-s)^{-\delta}\varphi_j'(s) {\rm d}s\right) \psi_m(t) {\rm d}t\right) u_j = f(t_m).
\end{equation*}
%The coefficients in the linear combination of the discrete solution $u$, that is, $u_1, \ldots, u_M$ are obtained as the solution of 
This leads to a linear system of equations $R\bm{u} = \bm{f}$, where $\bm{u} = (u_1, u_2, \cdots, u_M)^T$, $\bm{f}=(f(t_1), f(t_2), \cdots, f(t_M))^T$, and the entries of the coefficient matrix are 
\begin{equation}\label{elements_matrix_R}
R_{m,j} =  \frac{1}{\Gamma(1-\delta)}\int_0^T \left( \int_0^t (t-s)^{-\delta}\varphi_j'(s) {\rm d}s\right) \psi_m(t) {\rm d}t.
\end{equation}
It is easy to see that when $j>m$ $$R_{m,j} = \frac{1}{\Gamma(1-\delta)}\int_0^{t_m} (t_m-s)^{-\delta}\varphi_j'(s) {\rm d}s = 0,$$ 
and, therefore, $R$ is a dense lower triangular matrix whose entries are given by $R_{m, m-k} = d_{m, k+1} - d_{m,k}$ for $k=0,\ldots,m-1$, where 
\begin{eqnarray*}
d_{m,0}&=&0,\\
d_{m,1} &=& \frac{\tau_m^{-\delta}}{\Gamma(2-\delta)}, \\
d_{m,k} &=& \frac{1}{\Gamma(2-\delta)} \left(\frac{(t_m-t_{m-k})^{1-\delta}-(t_m-t_{m-k+1})^{1-\delta}}{\tau_{m-k+1}} \right), \ k = 2,\ldots,m-1.
\end{eqnarray*}
Note that the obtained discretization of the Caputo derivative is
\begin{equation*}  %\label{caputo_discrete}
D_M^{\delta} u_{m} = \frac{1}{\Gamma(2-\delta)}\left(d_{m,1}u_m +\sum_{k = 1}^{m-1} (d_{m,k+1}-d_{m,k}) u_{m-k}  \right),
\end{equation*}
which corresponds to the generalization of the well-known L1 scheme~\cite{lin_xu, Oldham_Spanier, ZhongSun_Wu} for non-uniform grids~\cite{stynes_graded}. 

\begin{remark} By considering different choices of the test or trial functions, we can obtain different discretizations for the time-fractional Caputo derivative.  %The method proposed in this work can be generalized to those discretizations.   
For example, by using piecewise quadratic basis functions for trial functions, we can obtain a temporal discretization with higher accuracy.  There are many high order schemes for the time fractional problems, see e.g.,~\cite{Alikhanov.A2015a,Cao.J;Xu.C2013a,Zhao.X;Sun.Z;Karniadakis.G2015a}, and it will be interesting to see if we can reconstruct some of the existing high order schemes by choosing different test and trial functions. In this work, however, we focus on the L1 scheme and the study of high order schemes will be the subject of our future work. 

\end{remark}

Finally, by denoting as $u_{n,m}$ the nodal approximation to the solution at each grid point $(x_n,t_m)$, we
approximate~\eqref{model_IVP_1}-\eqref{model_IVP_3} by the following discrete problem,
\begin{align}
&D_M^{\delta} u_{n,m} - \frac{u_{n+1,m}-2u_{n,m}+u_{n-1,m}}{h^2} = f(x_n,t_m), \; 1\leq n\leq N, \; 1\leq m\leq M, \label{discrete_model_IVP_1} \\
&u_{0,m} = 0,\; u_{N+1,m}=0, \quad 0<m\leq M, \label{discrete_model_IVP_2} \\
&u_{n,0} = g(x_n), \quad 0\leq n\leq N+1. \label{discrete_model_IVP_3}
\end{align}
In the case of a temporal uniform grid with time-step $\tau$, it is proved rigorously that this scheme has a rate of convergence of ${\cal O}(h^2+\tau^{\delta})$ for ``typical'' solutions of the time-fractional heat equation, that is,  solutions presenting a boundary layer at the initial time. To improve the poor convergence when the fractional order $\delta$ is very small, in \cite{stynes_graded} the authors proposed to use a graded mesh in time for which the resulting scheme has a convergence order of ${\cal O}(h^2+M^{-(2-\delta)})$. 
Such a grid is defined by $t_m = T(m/M)^r$ with $m=0,\ldots, M$, where $r\geq (2-\delta)/\delta$. In particular, in this work we choose $r= (2-\delta)/\delta$, as suggested in \cite{stynes_graded}.

Notice that matrix $R$ is a dense lower triangular matrix.  For the case of a uniform grid, it has Toeplitz structure, which can be exploited to develop efficient algorithms. For example, a geometric multigrid (GMG) method with a computational cost of ${\cal O}(NM\log(M))$ was proposed in~\cite{SISC_Gaspar}.  For the case of a non-uniform grid, $R$ does not have Toeplitz structure anymore and although the GMG method~\cite{SISC_Gaspar} can still be applied, this gives rise to a significant increase of the computational complexity.  Our aim is to approximate the matrix $R$ by $\widetilde{R}$ which can be stored in a data-sparse format and then design an efficient multigrid solver with a computational cost of ${\cal O}(kNM\log(M))$ where $k$ is a parameter that controls the accuracy of such approximation. 

%%%%%%%%%%%%%%%%%%%%%%%%%%%%%%%%%%%%%%%%%%%%%%%%%

\section{Discretization based on ${\cal H}$-Matrices representation}\label{sec:3}

In this section we aim to approximate matrix $R$ by $\widetilde{R}$ based on the ${\cal H}$-matrices framework.  According to~\eqref{elements_matrix_R}, the entries of $R$ are defined using the kernel~${\cal K}(t,s) =(t-s)^{-\delta}$.  Similar to typical kernel functions, singularities only occur when $t = s$.  The kernel function is smooth everywhere else and decays when $|t-s| \rightarrow \infty$.  This implies that the entries of $R$ decay to $0$ when they are far away from the diagonal, and then they usually can be replaced by low-rank approximations.  Such approximation can be done by replacing the original kernel by a {\it degenerate (separable) kernel},
\begin{equation}\label{degenerate_kernel}
\widetilde{\cal K}(t,s) := \sum_{\nu = 0}^{k-1} p_{\nu}(t) q_{\nu}(s).
\end{equation}
This is a partial sum of $k$ terms which usually is obtained by truncation of certain infinite sum of $\mathcal{K}$.  Moreover, each term is a product of two functions, one only depends on $t$ and the other one only depends on $s$.  This approximation is appropriate, however, only when the truncation error can be bounded uniformly, i.e., $t$ and $s$ should be sufficiently far away from each other.  More concretely, the feasibility of such approximation is characterized by the following condition.
\begin{definition}[Admissibility condition]\label{admissibility}
Let $I_t:=[a,b],\ I_s:=[c,d]\subset [0,T]$ be two intervals such that $d<a$. We say that $I_t\times I_s$ satisfies the {\it admissibility condition} if the following holds,
\begin{equation}\label{def:admissibility}
\hbox{{\rm diam}}(I_t) \leq \hbox{{\rm dist}}(I_s,I_t).
\end{equation}
%that is, $d-c\leq a-d$.
We define the set of indices $\vartheta \times \sigma = \{ (m,j) \  \ni \ (\supp \psi_m,\supp \varphi_j)\subset I_t\times I_s\}$. We say that $\vartheta \times \sigma$ is admissible if $I_t\times I_s$ satisfies the admissibility condition.
\end{definition}

\begin{remark}
More general admissibility conditions could be considered as shown in~\cite{Bebendorf.M2008a,H_matrices_Hackbusch_book}.  Here, we use this simple choice to demonstrate the idea and make the presentation easy to follow.
\end{remark}

Next, we approximate the kernel ${\cal K}(t,s)$ by its truncated Taylor expansion in a set $I_t\times I_s$ satisfying the admissibility condition. Taking into account that
$$\partial^{\nu}_t\left[(t-s)^{-\delta}\right] = (-1)^{\nu}\prod_{l=1}^{\nu}(\delta + l -1)(t-s)^{-(\delta+\nu)},\ \nu=1,2,\ldots,$$
the truncated Taylor expansion of the kernel ${\cal K}(t,s)$ at the midpoint of interval $I_t$, that is $t_0 = \frac{a+b}{2}$, is given by
\begin{equation*} %\label{taylor_exp_kernel}
{\cal K}(t,s) \approx \sum_{\nu =0}^{k-1}\frac{1}{\nu !}\left( \prod_{l = 1}^{\nu} (1-\delta-l)\right) (t_0-s)^{-(\delta+\nu)}(t-t_0)^{\nu} =: \widetilde{\cal K}(t,s).
\end{equation*} 
In this way, we have obtained a degenerate kernel $\widetilde{\cal K}(t,s)$ in $I_t\times I_s$ \eqref{degenerate_kernel} where 
\begin{eqnarray*}
p_{\nu}(t) &:=& (t-t_0)^{\nu},\\
q_{\nu}(s) &:=& \frac{1}{\nu !}\left( \prod_{l = 1}^{\nu} (1-\delta-l)\right) (t_0-s)^{-(\delta+\nu)}.
\end{eqnarray*}
By using the degenerate kernel, we approximate the matrix entries $R_{m,j}$ in \eqref{elements_matrix_R}, $\forall (m,j)\in \vartheta \times \sigma$, by the following
\begin{eqnarray*}
R_{m,j} \approx \widetilde{R}_{m,j} &=&  \frac{1}{\Gamma(1-\delta)}\int_0^T \left( \int_0^t \widetilde{\cal K}(t,s) \varphi_j'(s) {\rm d}s\right) \psi_m(t) {\rm d}t \\
&=& \frac{1}{\Gamma(1-\delta)}\int_0^T \left( \int_0^t  \sum_{\nu = 0}^{k-1} p_{\nu}(t) q_{\nu}(s) \varphi_j'(s) {\rm d}s\right) \psi_m(t) {\rm d}t.
\end{eqnarray*}
We observe that the double integral can be separated into a product of two single integrals as follows,
\begin{equation*}
\widetilde{R}_{m,j} = \frac{1}{\Gamma(1-\delta)} \sum_{\nu = 0}^{k-1} \left(\int_0^Tp_{\nu}(t)  \psi_m(t) {\rm d}t\right) \left(\int_0^t  q_{\nu}(s) \varphi_j'(s) {\rm d}s\right).
\end{equation*}
From this expression, submatrix $R|_{\vartheta\times \sigma}$ can be approximated by $\widetilde{R}|_{\vartheta\times \sigma}$, which is a ${\bm R}k$-matrix \cite{H_matrices_Hackbusch}, i.e., (see Figure \ref{low_rank_pic})
$$R|_{\vartheta\times \sigma} \approx \widetilde{R}|_{\vartheta \times \sigma} = A B^T, \quad A\in {\mathbb R}^{|\vartheta| \times k}, \ B\in{\mathbb R}^{|\sigma|\times k},$$
where $|\vartheta|$ and $|\sigma|$ denote the cardinality of sets $\vartheta$ and $\sigma$, respectively.  The entries of the matrices $A$ and $B$ are given by
\begin{eqnarray}
&& \! \! \! \! \! \! \! \! \! \! \! \! \! \! \! \! A_{m,\nu} \! \! := \! \! \frac{1}{\Gamma(1-\delta)}\int_0^Tp_{\nu}(t)  \psi_m(t) {\rm d}t \! = \! \frac{1}{\Gamma(1-\delta)} (t_m-t_0)^{\nu}, \label{A_elements} \\
&& \! \! \! \! \! \! \! \! \! \! \! \! \! \! \! \!  B_{j,\nu} \! \! := \! \! \int_0^t \! \!  q_{\nu}(s) \varphi_j'(s) {\rm d}s \! = \! \frac{c_{\nu}}{1-\delta-\nu}
\! \! \left(\!\frac{1}{\tau_{j-1}}(m_{j-1}-m_j) -\frac{1}{\tau_{j}}(m_j-m_{j+1})\! \! \right)\! \!, 
\label{B_elements}
\end{eqnarray}
where $m_i = (t_0-t_i)^{1-\delta-\nu}$ and $\displaystyle c_{\nu} = \frac{1}{\nu !}\left( \prod_{l = 1}^{\nu} (1-\delta-l)\right)$.
\begin{figure}[htb]
\begin{center}
\includegraphics[width = 0.4\textwidth]{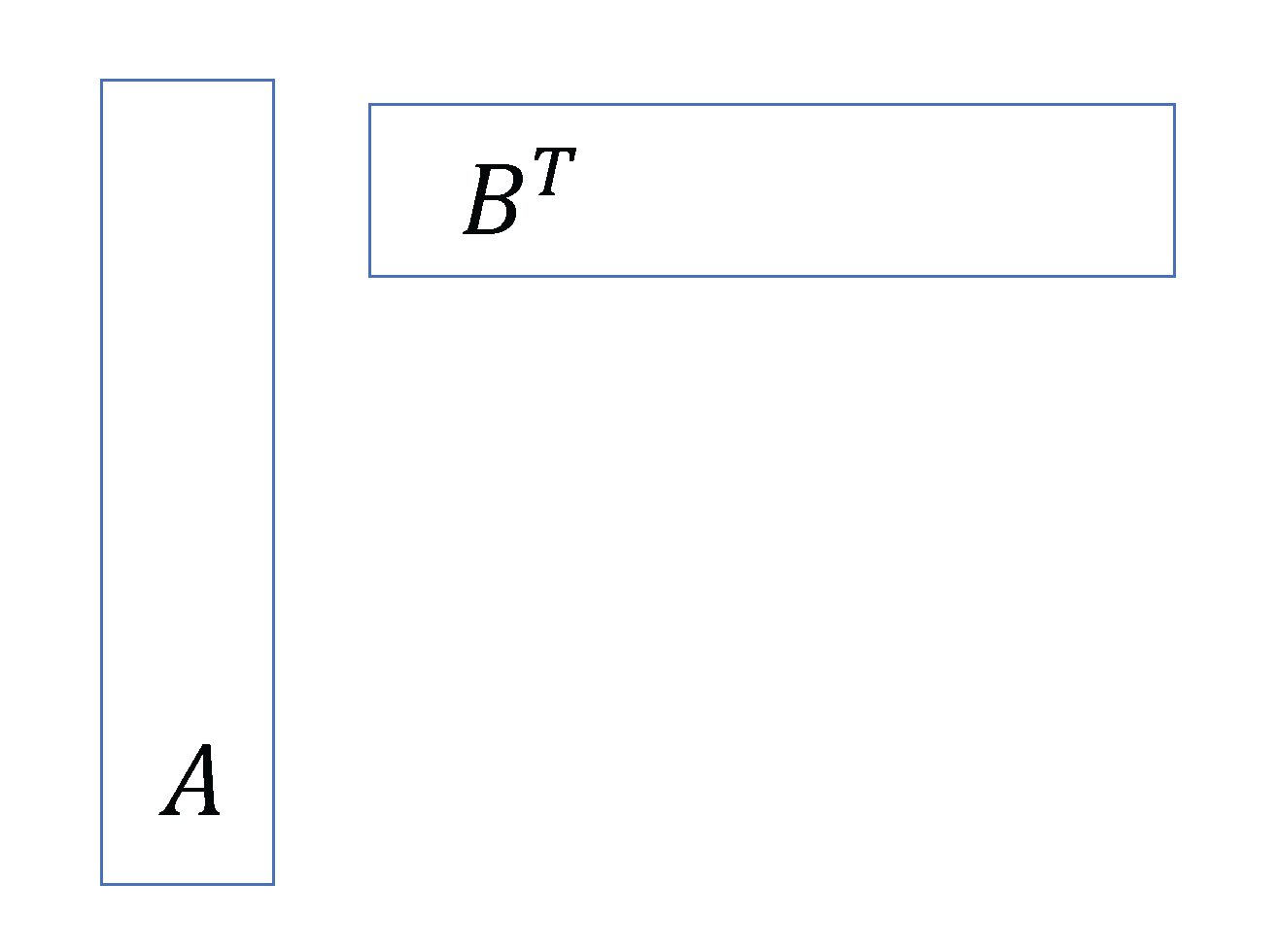}
\caption{Representation of ${\bm R}k$-matrix in factorized form.}
\label{low_rank_pic}
\end{center}
\end{figure}

Notice that submatrix $R|_{\vartheta \times \sigma}$ is approximated by a low-rank representation with at most rank $k$, which only needs $k(|\vartheta|+|\sigma|)$ number of elements to store.  More importantly, this approximation in an admissible block gives rise to a uniformly bounded element-wise truncation error, as stated in the following theorem.
\begin{theorem}\label{theorem_estimate}
Let $I_t:=[a,b],\ I_s:=[c,d]$ be two intervals on $[0,T]$ such that $d<a$. Assume that $I_t\times I_s$ satisfies the admissibility condition~\eqref{def:admissibility}. Then, for all set of indices $(m,j)\in \vartheta\times \sigma$, we have the following estimate of the approximation error,
\begin{equation*} %\label{error_estimate}
|R_{m,j}- \widetilde{R}_{m,j}| = {\cal O}(3^{-k}),
\end{equation*}
where $k$ is the number of terms considered in the truncated Taylor expansion of the kernel ${\cal K}(t, s)$.
\end{theorem}
\begin{proof}
From \eqref{A_elements} it is easy to see that
\begin{equation}\label{bound_A}
|A_{m,\nu}| \leq \frac{1}{\Gamma(1-\delta)}|a-t_0|^{\nu},\ \ \nu\geq 0,
\end{equation}
where $t_0=(a+b)/2$. Taking into account the following Taylor expansions,
\begin{eqnarray*}
(t_0-t_{j-1})^{1-\delta-\nu} &=& (t_0-t_j)^{1-\delta-\nu}+\tau_{j-1}(1-\delta-\nu)(t_0-t_j)^{-\delta-\nu}\\
&+& \frac{\tau_{j-1}^2}{2}(1-\delta-\nu)(-\delta-\nu)(t_0-\xi_j)^{-1-\delta-\nu}, \ \ \xi_j\in[t_{j-1},t_j],\\
(t_0-t_{j+1})^{1-\delta-\nu} &=& (t_0-t_j)^{1-\delta-\nu}-\tau_{j}(1-\delta-\nu)(t_0-t_j)^{-\delta-\nu}\\
&+& \frac{\tau_{j}^2}{2}(1-\delta-\nu)(-\delta-\nu)(t_0-\xi_{j+1})^{-1-\delta-\nu}, \ \ \xi_{j+1}\in[t_{j},t_{j+1}],
\end{eqnarray*}
from~\eqref{B_elements}, we obtain that
\begin{equation*} %\label{taylor_B}
B_{j,\nu} =  \frac{c_{\nu}(-\delta-\nu)}{2}\left(\tau_{j-1}(t_0-\xi_j)^{-(1+\delta+\nu)}+\tau_j(t_0-\xi_{j+1})^{-(1+\delta+\nu)}\right).
\end{equation*}
By using that $|c_{\nu}|\leq 1$ and that $\xi_j, \xi_{j+1}\in [c,d]$, we have the following bound
\begin{equation}\label{bound_B}
|B_{j,\nu}|\leq \frac{\delta+\nu}{2} (\tau_{j-1}+\tau_j) \left(\frac{1}{|a-t_0|+|a-d|}\right)^{1+\delta+\nu}, \ \ \nu\geq 0.
\end{equation}
For each $\nu\geq 0$, by using the bounds in \eqref{bound_A} and \eqref{bound_B}, the product of $A_{m,\nu}$ and $B_{j,\nu}$ is given as follows,
\begin{equation*}
|A_{m,\nu} B_{j,\nu}| \leq  \frac{(\delta+\nu) (\tau_{j-1}+\tau_j)}{2\Gamma(1-\delta)(|a-t_0|+|a-d|)^{1+\delta}}  \left(\frac{|a-t_0|}{|a-t_0|+|a-d|}\right)^{\nu}.
\end{equation*}
Denoting $\displaystyle r = \frac{|a-t_0|}{|a-t_0|+|a-d|}$, we can bound the element-wise error in the following way,
\begin{equation}\label{bound_error_aux}
|R_{m,j}-\widetilde{R}_{m,j}| \leq \sum_{\nu = k}^{\infty} |A_{m,\nu} B_{j,\nu}|\leq  \frac{(\tau_{j-1}+\tau_j)}{2\Gamma(1-\delta)(|a-t_0|+|a-d|)^{1+\delta}}  \sum_{\nu = k}^{\infty} (\delta+\nu) r^{\nu}.
\end{equation}
The sum of the series appearing in \eqref{bound_error_aux} is given by,
$$\sum_{\nu = k}^{\infty} (\delta+\nu) r^{\nu} = r^k \left(\frac{\delta}{1-r} + \frac{k}{1-r} + \frac{r}{(1-r)^2}\right).$$
Taking into account that $\displaystyle r = \frac{{\rm diam}(I_t)}{{\rm diam}(I_t) + 2{\rm dist}(I_t,I_s)}$ and that $\displaystyle \frac{{\rm diam}(I_t)}{{\rm dist}(I_t,I_s)}\leq 1$ due to the admissibility condition~\eqref{def:admissibility}, we have 
%the following bound of the sum in \eqref{bound_error_aux},
\begin{equation*} %\label{bound_sum}
\sum_{\nu = k}^{\infty} (\delta+\nu) r^{\nu} \leq 3^{-k}\left(\frac{3}{2} (\delta+k)+ \frac{3}{4}\right).
\end{equation*}
Note that the dominating term is $3^{-k}$, and thus we can conclude that 
$$|R_{m,j}-\widetilde{R}_{m,j}|  = {\cal O}(3^{-k}),$$
which gives us an upper bound of the element-wise truncation error. 
\end{proof}
\textcolor{red}{
\begin{remark} \label{rem:super-close}
Using the $\mathcal{H}$-matrix representation means we are solving the perturbed linear system $\widetilde{R}\widetilde{\bm{u}} = \bm{f}$, $\widetilde{\bm{u}} = (\widetilde{u}_1, \widetilde{u}_2, \cdots, \widetilde{u}_M)^T$.  By the standard perturbation theory of solving linear systems of equations and the element-wise error estimate $|R_{m,j}-\widetilde{R}_{m,j}|  = {\cal O}(3^{-k})$ from Theorem~\ref{theorem_estimate}, we can easily get that $\| \bm{u} - \widetilde{\bm{u}} \|_{\infty} = \mathcal{O}(3^{-k})$, where $\|\cdot \|_{\infty}$ denotes the standard $\ell^{\infty}$-norm.  By triangular inequality, we naturally have 
\begin{align*}
\max_{1\leq j \leq M} |u(t_j) - \widetilde{u}_j | &\leq \max_{1 \leq j \leq M} |u(t_j) - u_j | + \max_{1 \leq j \leq M} |u_j - \widetilde{u}_j | \\
& = \max_{1 \leq j \leq M} |u(t_j) - u_j | + \mathcal{O}(3^{-k}).
\end{align*}
Therefore, if we choose $k$ large enough, the error introduced by the $\mathcal{H}$-matrix representation is negligible and the error estimate remains the same, i.e., $\mathcal{O}(\tau^{\delta})$ on a temporal uniform grid and $\mathcal{O}(M^{-(2-\delta)})$ on a graded mesh in time.  For example, we choose $k=20$ in our experiments, since $3^{-20} \approx 3\times10^{-10}$ is quite small, we can see that the convergence order of the numerical solution obtained by the $\mathcal{H}$-matrix representation remains the same numerically. 
\end{remark}
}

The approximation of the admissible blocks of matrix $R$ gives its ${\cal H}$-matrix representation $\widetilde{R}$, in which those admissible submatrices are stored in a ${\bm R}k$-matrix representation and the rest submatrices are stored in a full-matrix format.  In practice, such an ${\cal H}$-matrix representation is usually constructed by using tree data structures. For the details of the construction we refer the reader to the books \cite{Bebendorf.M2008a,H_matrices_Hackbusch_book}.  Here, we only present an intuitive explanation of the algorithm for constructing the ${\cal H}$-matrix representation.  We start from the original coefficient matrix, by dividing it into four submatrices, as shown in Figure \ref{construction_Hmat} (a). The upper right block is a zero block because of the structure of $R$ and it is colored in light blue. Then, we check if the other subblocks satisfy the admissibility condition~\eqref{def:admissibility}.  At this level, none of them satisfies the admissibility condition and each of them is recursively divided into four blocks, yielding to the structure shown in Figure \ref{construction_Hmat} (b).  The blocks above the diagonal are zero again and we check the admissibility condition for the rest of the blocks.  For this level, three blocks are admissible and they are stored as low-rank approximation (colored in purple in the picture).  The rest of the blocks (colored in pink) are split again recursively.  In this way, and by repeating the same checking and coloring procedure, we obtain the structure in Figure \ref{construction_Hmat} (c). This process is recursively carried out until we get the resulting ${\cal H}$-matrix representation, in which the non-zero subblocks satisfying the admissibility condition are stored in a low-rank matrix representation and the rest of the blocks are stored in a full-matrix representation.
\begin{figure}[htb]
\begin{center}
\begin{tabular}{ccc}
\includegraphics[width = 0.3\textwidth]{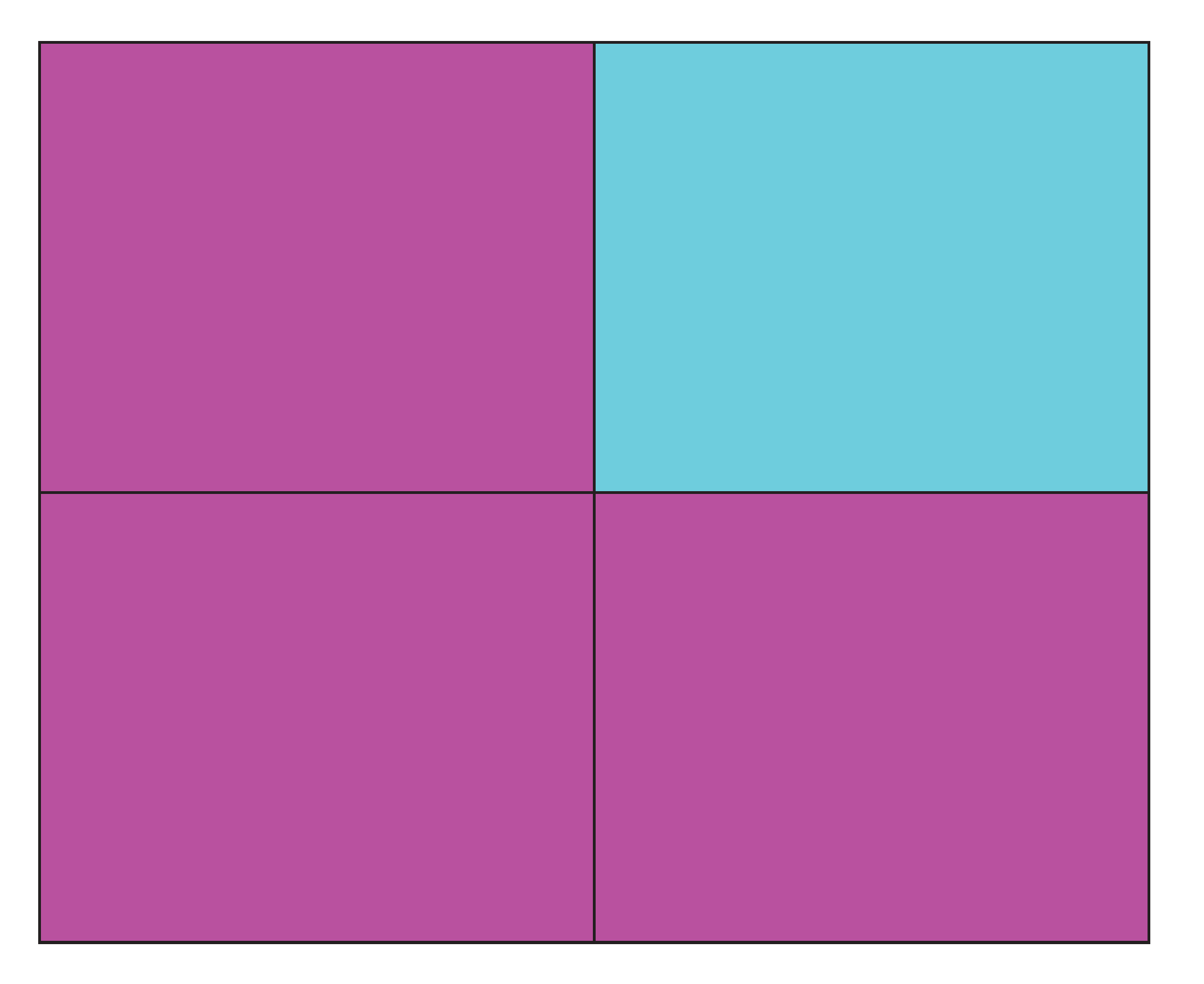}
&
\includegraphics[width = 0.3\textwidth]{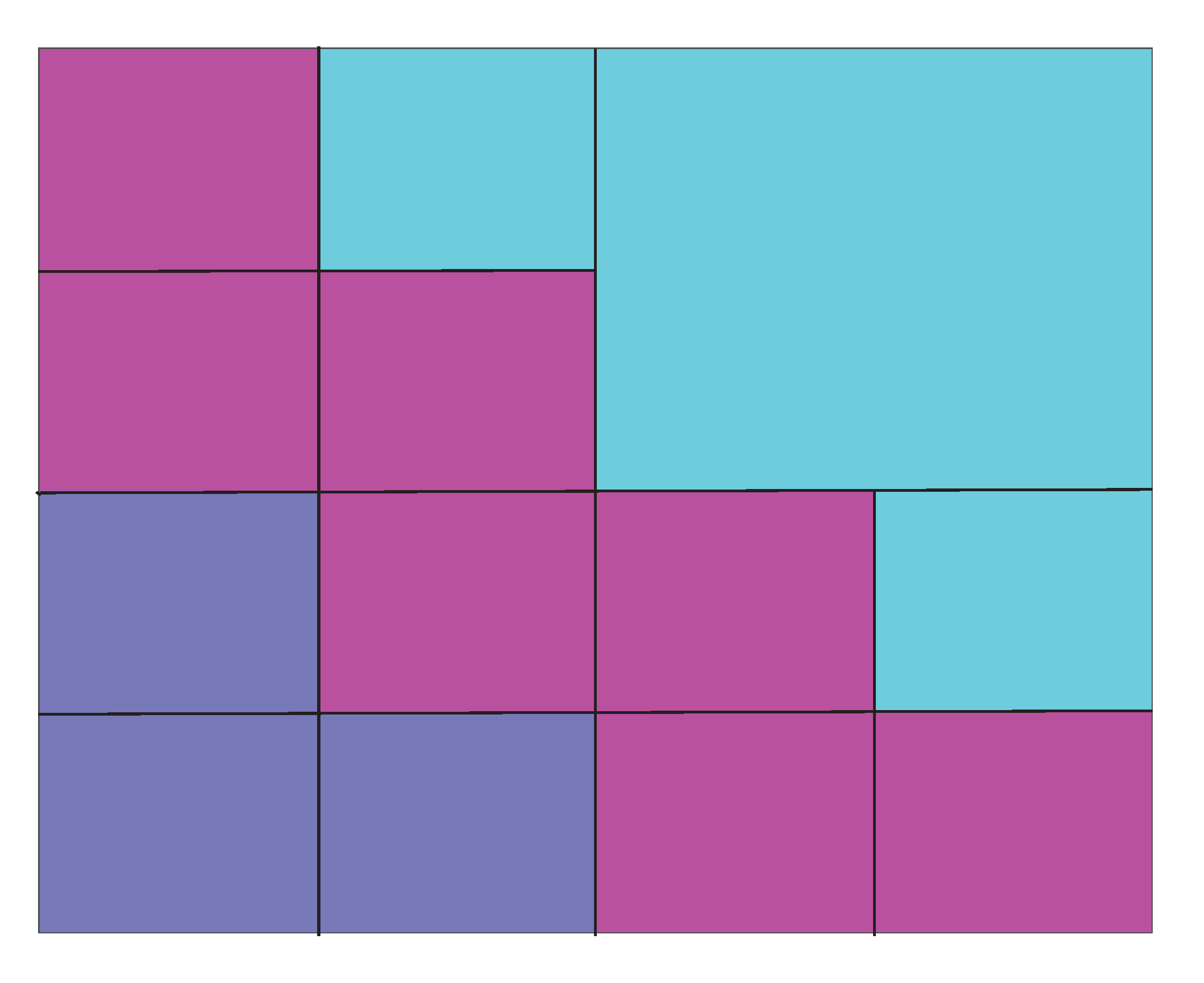}
&
\includegraphics[width = 0.3\textwidth]{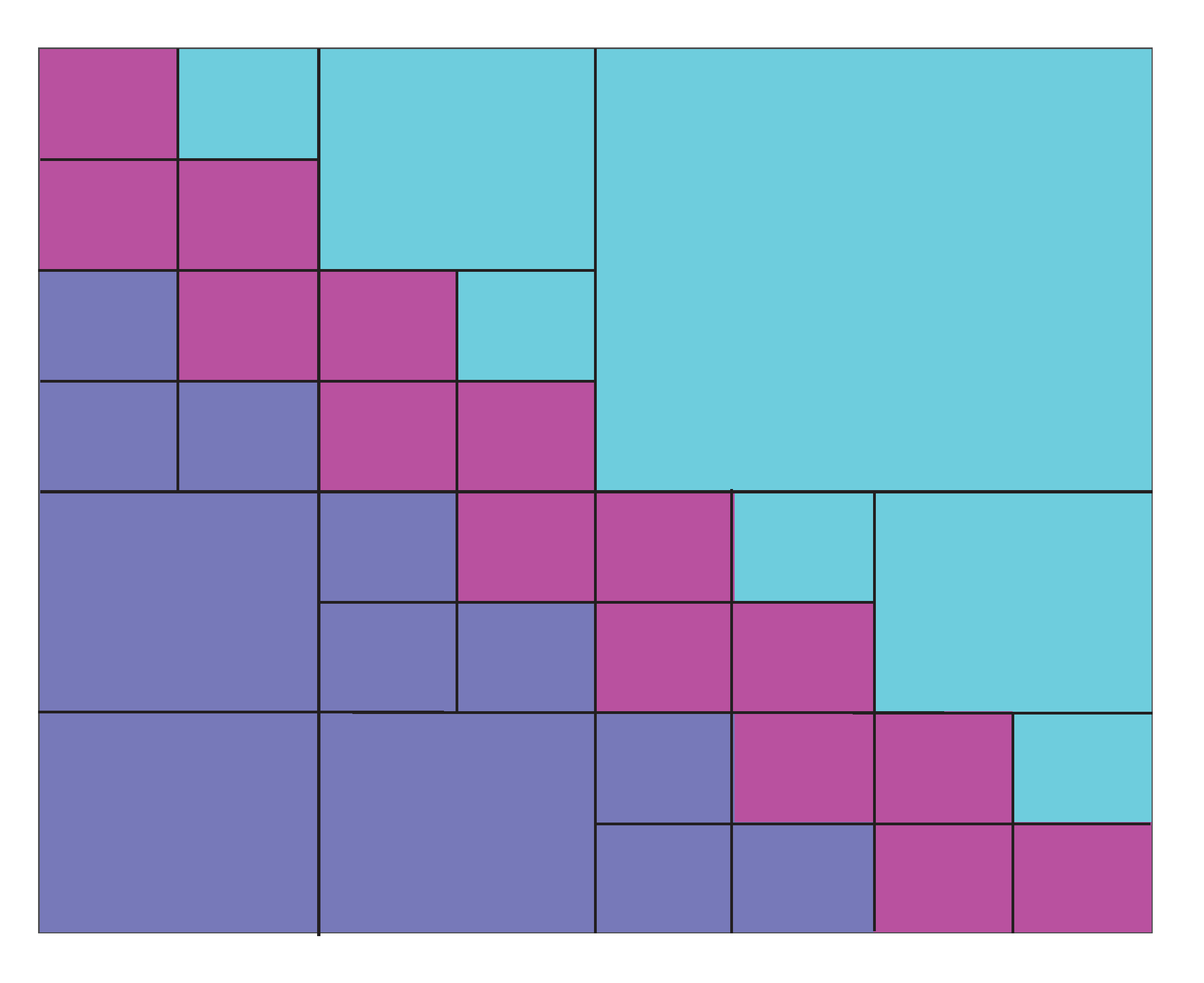}
\\
(a) & (b) & (c)
\end{tabular}
\caption{Schematic procedure for the construction and storage of the ${\cal H}$-matrix representation. Zero blocks are colored in light blue, low-rank approximation blocks in purple, and those blocks stored in a full-matrix representation in pink.}
\label{construction_Hmat}
\end{center}
\end{figure}

We would like to emphasize that the ${\cal H}$-matrix representation is computed only for the matrix corresponding to the time-discretization, i.e., matrix $R$, which is a dense matrix. Combining this representation $\widetilde{R}$ with the spatial discretization matrix $A_h$ in \eqref{system_ODEs}, we obtain the coefficient matrix of the fully-discrete system, denoted here by $A_{h,\tau}$.

%%%%%%%%%%%%%%%%%%%%%%%%%%%%%%%%%%%%%%%%%%%%%%%%%%

\section{Multigrid waveform relaxation method}\label{sec:4}  

In this section, we introduce the proposed multigrid waveform relaxation method for solving~\eqref{discrete_model_IVP_1}-\eqref{discrete_model_IVP_3} based on the ${\cal H}$-matrix representation, which gives an optimal algorithm even on non-uniform grids.

Waveform relaxation methods, also known as dynamic iteration methods, are continuous-in-time iterative algorithms for numerically solving large systems of ordinary differential equations.  Different from standard iterative techniques, waveform relaxation iterates functions in time instead of scalar values. 
%Their difference with respect standard iterative techniques lies in the fact that their iterates consist of functions in time instead of scalar values. 
% Waveform relaxation methods are also called dynamic iteration methods while the corresponding relaxation methods for systems of algebraic equations are then referred to as static iterations.
They can also be applied to solve time dependent PDEs by replacing the spatial derivatives by the spatial discrete approximations, obtaining semi-discretizations of the problems. In this way, the PDEs are transformed into a large set of ordinary differential equations, which can then be solved by an iterative algorithm. 

For the time-fractional heat equation, after semi-discretization, we obtain the system of ordinary differential equations \eqref{system_ODEs}.
Next step is to solve~\eqref{system_ODEs} by an iterative method.  In particular, in this work, we consider a red-black Gauss-Seidel iteration (denoted by $S_h$) which consists of a two-stage procedure, i.e., the updates are performed first on the even points and then on the grid-points with odd numbering.  In addition, to accelerate the convergence of the red-black Gauss-Seidel waveform relaxation, a coarse-grid correction procedure based on a coarsening strategy only in the spatial dimension is performed. This results in the so-called linear multigrid waveform relaxation algorithm~\cite{Vandewalle_book}.  If standard inter-grid transfer operators, as full-weighting restriction and linear interpolation, are considered, the algorithm of the multigrid waveform relaxation (WRMG) is given in Algorithm~\ref{wrmg_1}, where $k$ denotes the iteration number and $t$ is the time variable.
\begin{algorithm}[H]
\caption{ \textbf{: Multigrid waveform relaxation: ${\mathbf{u_{h}^k(t) \rightarrow u_{h}^{k+1}(t)}}$}}\label{wrmg_1}
\vspace{0.3cm}
\begin{algorithmic}
\IF{we are on the coarsest grid-level (with spatial grid-size given by $h_0$)}
%\STATE
\STATE $D_t^{\delta} u_{h_0}^{k+1}(t) + A_{h_0}u_{h_0}^{k+1}(t)= f_{h_0}(t)$ \hspace{0.9cm} {Solve with a direct or fast solver.}
%\STATE
\ELSE
\STATE {
\begin{tabular}{lr}
\\ [-3.3ex]
$\overline{u}_{h}^k(t)=S_{h}^{\nu_1}(u_{h}^k(t))$ & \hspace{-3.5cm}\textbf{(Pre-smoothing)}\\ & \hspace{-4.3cm}$\nu_1$ steps of the \textbf{red-black waveform relaxation}.\\
\\ [-3.3ex]
$\overline{r}_{h}^k(t)=f_{h}(t) - (D_t^{\delta} + A_{h})\, \overline{u}_{h}^k(t)$ & \hspace{-3.5cm}Compute the defect. \\
\\ [-3.3ex]
$\overline{r}_{2h}^k(t)=I_h^{2h}\, \overline{r}_{h}^k(t)$ & \hspace{-3.5cm}Restrict the defect. \\
\\ [-3.3ex]
$(D_t^{\delta}+A_{2h}) \widehat{e}_{2h}^k(t) = \bar{r}_{2h}^k(t),\; \widehat{e}_{2h}^k(0) = 0$ & \hspace{-3.5cm}Solve the defect equation \\ [0.5ex]
& \hspace{-3.5cm} on $G_{2h}$ by performing $\gamma \ge 1$ cycles of WRMG. \\
\\ [-3.3ex]
$\widehat{e}_{h}^k(t) = I_{2h}^h \, \widehat{e}_{2h}^k(t) $ & \hspace{-3.5cm}Interpolate the correction. \\
\\ [-3.3ex]
$\overline{u}_{h}^{k+1}(t) = \overline{u}_{h}^k(t) + \widehat{e}_{h}^k(t)$ & \hspace{-3.5cm}Compute a new approximation. \\
\\ [-3.3ex]
$u_{h}^{k+1}(t)=S_{h}^{\nu_2}(\overline{u}_{h}^{k+1}(t))$ & \hspace{-3.5cm}\textbf{(Post-smoothing)}\\ & \hspace{-4.3cm}$\nu_2$ steps of the \textbf{red-black waveform relaxation}. \\
\\ [-3.3ex]
\end{tabular}}
\ENDIF
\end{algorithmic}
\vspace{0.3cm}
\end{algorithm}

In the practical implementation of Algorithm \ref{wrmg_1}, a discrete-time algorithm should be used.  Therefore, after discretizing in time by replacing the differential operator $D_t^{\delta}$ by $D_M^{\delta}$, the previous algorithm can be interpreted as a space-time multigrid method with coarsening only in space.   In order to reduce the high computational cost of solving $D_M^{\delta}$, we use the ${\cal H}$-matrix representation of operator  $D_M^{\delta}$.  Therefore, the algorithm is applied to the coefficient matrix ${\cal A}_{h,\tau}$ which is obtained by the spatial discretization matrix $A_h$ and the ${\cal H}$-matrix representation $\widetilde{R}$ in time.  Finally, the whole multigrid waveform relaxation combines a zebra-in-time line relaxation with a standard semi-coarsening strategy only in the spatial dimension as shown in Algorithm~\ref{wrmg}, where $k$ denotes the iteration number.
\begin{algorithm}[H]
\caption{ \textbf{: Multigrid waveform relaxation: ${\mathbf{u_{h,{\boldsymbol \tau}}^k \rightarrow u_{h,{\boldsymbol \tau}}^{k+1}}}$}}\label{wrmg}
\vspace{0.3cm}
\begin{algorithmic}
\IF{we are on the coarsest grid-level (with grid-size given by $h_0$ and $\tau$)}
%\STATE
\STATE $u_{h_0,\tau}^{k+1} = \mathcal{A}_{h_0,\tau}^{-1}f_{h_0,\tau}$ \hspace{2.9cm} {Solve with a direct or fast solver.}
%\STATE
\ELSE
\STATE {
\begin{tabular}{lr}
\\ [-3.3ex]
$\overline{u}_{h,\tau}^k=S_{h,\tau}^{\nu_1}(u_{h,\tau}^k)$ & \textbf{(Pre-smoothing)}\\ & $\nu_1$ steps of \textbf{zebra-in-time line relaxation}.\\
\\ [-3.3ex]
$\overline{r}_{h,\tau}^k=f_{h,\tau} - \mathcal{A}_{h,\tau}\, \overline{u}_{h,\tau}^k$ & Compute the defect. \\
\\ [-3.3ex]
$\overline{r}_{2h,\tau}^k=I_h^{2h}\, \overline{r}_{h,\tau}^k$ & Restrict the defect \textbf{only in space}. \\
\\ [-3.3ex]
$\widetilde{e}_{2h,\tau}^k = 0$ & Take the zero grid function as a first  \\
& approximation on the coarse grid. \\
\\ [-3.3ex]
$\mathcal{A}_{2h,\tau} \widehat{e}_{2h,\tau}^k = \bar{r}_{2h,\tau}^k$ & Solve the defect equation on $\Omega_{2h,\tau}$ \\
& by performing $\gamma \ge 1$ cycles of WRMG. \\
\\ [-3.3ex]
$\widehat{e}_{h,\tau}^k = I_{2h}^h \, \widehat{e}_{2h,\tau}^k $ & Interpolate the correction \textbf{only in space}. \\
\\ [-3.3ex]
$\overline{u}_{h,\tau}^{k+1} = \overline{u}_{h,\tau}^k + \widehat{e}_{h,\tau}^k$ & Compute a new approximation. \\
\\ [-3.3ex]
$u_{h,\tau}^{k+1}=S_{h,\tau}^{\nu_2}(\overline{u}_{h,\tau}^{k+1})$ & \textbf{(Post-smoothing)}\\ & $\nu_2$ steps of \textbf{zebra-in-time line relaxation}. \\
\\ [-3.3ex]
\end{tabular}}
\ENDIF
\end{algorithmic}
\vspace{0.3cm}
\end{algorithm}

\subsection{Computational cost of the algorithm}
From Algorithm~\ref{wrmg}, it is clear that the most time-consuming parts of the multigrid waveform
relaxation method are the calculation of the residual and the relaxation step. These components require matrix-vector multiplications and the solution of dense lower triangular systems, respectively.  A standard implementation of these parts would give rise to a computational cost of at least  ${\cal O}(NM^2)$, whereas the remaining components of the algorithm can be performed with a computational cost proportional to the number of unknowns. 
As we are going to see next, we can reduce the computational cost of the algorithm to ${\cal O}(kNM\log(M))$ thanks to the use of the hierarchical matrices. 

In the calculation of the defect, a matrix-vector multiplication is needed. The calculations corresponding to the spatial discretization can be performed with a computational cost of ${\cal O}(NM)$.  For each spatial grid-point, however, the matrix-vector multiplication $\widetilde{R}x$ for a given vector $x$ is required.   This can be carried out by using the standard matrix-vector multiplication based on the ${\cal H}$-matrix format (see~\cite{H_matrices_Hackbusch_book}).  We take advantage of the lower triangular structure of matrix $\widetilde{R}$ to avoid the calculations corresponding to the zero upper triangular part, and use a slightly modified matrix-vector multiplication algorithm, which is given in Algorithm~\ref{H_product}.  Basically, an extra $\textbf{if}$ statement is added to handle the case of zero blocks.

\begin{algorithm}[H]
\caption{ \textbf{: Matrix-vector multiplication in ${\cal H}$-matrix format}}\label{H_product}
\vspace{0.5cm}
\begin{algorithmic}
\STATE $y = \verb"Hmatvec"(H,x)$
%\STATE
\IF{$H$ is a zero matrix}
%\STATE
\STATE $y = 0$
%\STATE
\ELSIF{$H$ is full matrix}
%\STATE
\STATE $y = Hx$
%\STATE
\ELSIF{$H = AB^T$ ($H$ is low-rank approximation)}
%\STATE
\STATE $y = A(B^Tx)$
%\STATE
\ELSIF{$H$ is stored in $(2\times 2)$-block form $H = \left( \begin{array}{cc} H_{11} & H_{12} \\ H_{21} & H_{22} \end{array}\right)$}
%\STATE
\STATE partition $x = \left( \begin{array}{c} x_1 \\ x_2 \end{array}\right)$
%\STATE
\STATE $y_1 = $\verb"Hmatvec"$(H_{11},x_1)$ + \verb"Hmatvec"$(H_{12},x_2)$ 
%\STATE
\STATE $y_2 = $\verb"Hmatvec"$(H_{21},x_1)$ + \verb"Hmatvec"$(H_{22},x_2)$ 
%\STATE
\STATE $y = \left( \begin{array}{c} y_1 \\ y_2 \end{array}\right)$
%\STATE
\ENDIF
\end{algorithmic}
\vspace{0.5cm}
\end{algorithm}

It is well-known that the computational complexity for the matrix-vector multiplication in ${\cal H}$-matrix format is ${\cal O}(kM\log(M))$. Since this is the computational cost for each spatial grid-point, the whole matrix-vector product in the calculation of the residual requires ${\cal O}(kNM\log(M))$ operations. 

In the relaxation step, for each spatial grid-point we require the solution of a lower triangular system of $M$ equations, i.e., $(\widetilde{R}+2/h^2I_M)x = b$.  This is done by applying the standard forward substitution method in the ${\cal H}$-matrix format (Algorithm \ref{H_forward_substitution}) to $(\widetilde{R}+2/h^2I_M)$. It is well-known that such a method has a computational cost of ${\cal O}(kM\log(M))$, see e.g., \cite{Bebendorf.M2008a,H_matrices_Hackbusch_book}. 
\begin{algorithm}[H]
\caption{ \textbf{: Forward substitution in ${\cal H}$-matrix format}}\label{H_forward_substitution}
\vspace{0.5cm}
\begin{algorithmic}
\STATE $x = \verb"Hforwardsubst"(H,b)$
%\STATE
\IF{$H$ is full matrix}
%\STATE
\STATE $x = H^{-1}b$
%\STATE
\ELSIF{$H$ is stored in $(2\times 2)$-block form $H = \left( \begin{array}{cc} H_{11} & 0 \\ H_{21} & H_{22} \end{array}\right)$}
%\STATE
\STATE partition $b = \left( \begin{array}{c} b_1 \\ b_2 \end{array}\right)$
%\STATE
\STATE $x_1 = $\verb"Hforwardsubst"$(H_{11},b_1)$ 
%\STATE
\STATE $b_2 = b_2 - $\verb"Hmatvec"$(H_{21},x_1)$
\STATE $x_2 = $\verb"Hforwardsubst"$(H_{22},b_2)$  
%\STATE
\STATE $x = \left( \begin{array}{c} x_1 \\ x_2 \end{array}\right)$
%\STATE
\ENDIF
\end{algorithmic}
\vspace{0.5cm}
\end{algorithm}

Notice that in Algorithm \ref{H_forward_substitution}, we only need to consider the two cases
of the matrix being stored in a full-matrix format or in a $(2\times 2)$-block form because the diagonal blocks of $\widetilde{R}$ are always not admissible and stored in a dense matrix format.

%Notice that in Algorithm \ref{H_forward_substitution}, we only need to consider the cases either the matrix is stored in a full-matrix format or in a $(2\times 2)$-block form because the diagonal blocks of $\widetilde{R}$ are always not admissible and stored in a dense matrix format. 

Overall, the computational cost of an iteration of one $V$-cycle multigrid waveform relaxation method is ${\cal O}(kNM\log(M))$.  Since the convergence rate of such an algorithm usually is independent of the number of unknowns, as we will see, only few cycles are needed to reach the desired accuracy, and the total computational cost for solving the time-fractional heat equation on graded meshes is ${\cal O}(kNM\log(M))$.

\begin{remark} All the parallelization techniques for multigrid methods can be used for the implementation of the multigrid waveform algorithm  on parallel computers. Moreover, the proposed algorithm can be parallelized in the time direction by using the  multigrid waveform relaxation with cyclic reduction \cite{horton_vandewalle}.
\end{remark}

%%%%%%%%%%%%%%%%%%%%%%%%%%%%%%%%%%%%%%%%%%%%%%%%%%

\section{Numerical results}\label{sec:5}

This section is devoted to illustrate the good performance of the proposed multigrid waveform relaxation method for the solution of the time-fractional heat equation when graded meshes are considered.  We also present some results obtained by a semi-algebraic mode analyisis,
%local Fourier analysis, 
which usually provides accurate predictions of the performance of the multigrid methods, and indeed, it can be made rigorous if appropriate boundary conditions are considered.  In particular, here we consider the semi-algebraic mode analysis introduced in~\cite{sama}, which was already applied to study the convergence of the WRMG algorithm for the time-fractional heat equation in the case of uniform temporal grids in~\cite{SISC_Gaspar}.  This analysis is based on an exponential Fourier mode analysis or local Fourier analysis technique only in space and an exact analytical approach in time.  This is the key that allows us to apply this analysis when non-uniform grids in time are considered. For the details of this analysis we refer the readers to~\cite{SISC_Gaspar}.  The only modification that we need to do when using the graded meshes is to consider the new coefficients of the matrix corresponding to the time discretization.  Notice that the analysis is applied to the original discretization rather than to the ${\cal H}$-matrix representation that we use in the implementation.  As we proved in Theorem \ref{theorem_estimate}, the element-wise difference between both matrices is ${\cal O}(3^{-k})$ and, therefore, by choosing sufficiently large $k$, the results of the analysis by using the original matrix are reliable to predict the practical results obtained by using the ${\cal H}$-matrix representation (see Remark~\ref{rem:super-close}). We will see in the following experiments that this analysis gives rise to very accurate predictions.

We consider two numerical experiments: one test problem which is one dimensional in space, and a second example which is two dimensional in space. We choose rank $k=20$ in our experiments except in the cases where we vary $k$.  We will consider $V$-cycles for all the experiments since their convergence rates are similar to those obtained by $W$-cycles and, therefore, they provide a more efficient multigrid method in practice.

All numerical computations were carried out using MATLAB on a MacBook Pro with a
Core i5 2.7 GHz and 8 GB RAM, running OS X 10.10 (Yosemite).

\subsection{One-dimensional time-fractional heat equation}
In the first numerical experiment we consider a problem whose solution is smooth away from the initial time but has a certain singular behavior at $t=0$ where it presents an initial layer. These are reasonably general and realistic hypotheses on the behavior of the solution of the considered problems near the initial time. In particular, we consider problem~\eqref{model_IVP_1}-\eqref{model_IVP_3} defined on $[0,\pi]\times [0,1]$, with a zero right-hand side ($f(x,t) = 0$) and an initial condition $g(x) = \sin\, x$. The solution of this problem is $u(x,t) = E_{\delta}(-t^{\delta})\sin\, x$, where $E_{\delta}:{\mathbb R} \rightarrow {\mathbb R}$ is given by
$$E_{\delta}(z) := \sum_{k=0}^{\infty}\frac{z^k}{\Gamma(\delta \, k + 1)},$$
 (see \cite{Luchko, stynes_graded}). In Figure~\ref{solution_picture} (a), we can observe the singularity of the analytical solution for near the initial time, where an initial layer appears. The picture corresponds to the fractional order $\delta = 0.6$, and following the rule given in Section \ref{sec:2} to construct the optimal graded mesh, in Figure~\ref{solution_picture} (b) such a grid is displayed.  
\begin{figure}[htb]
\begin{center}
\begin{tabular}{cc}
\includegraphics[width =0.51\textwidth]{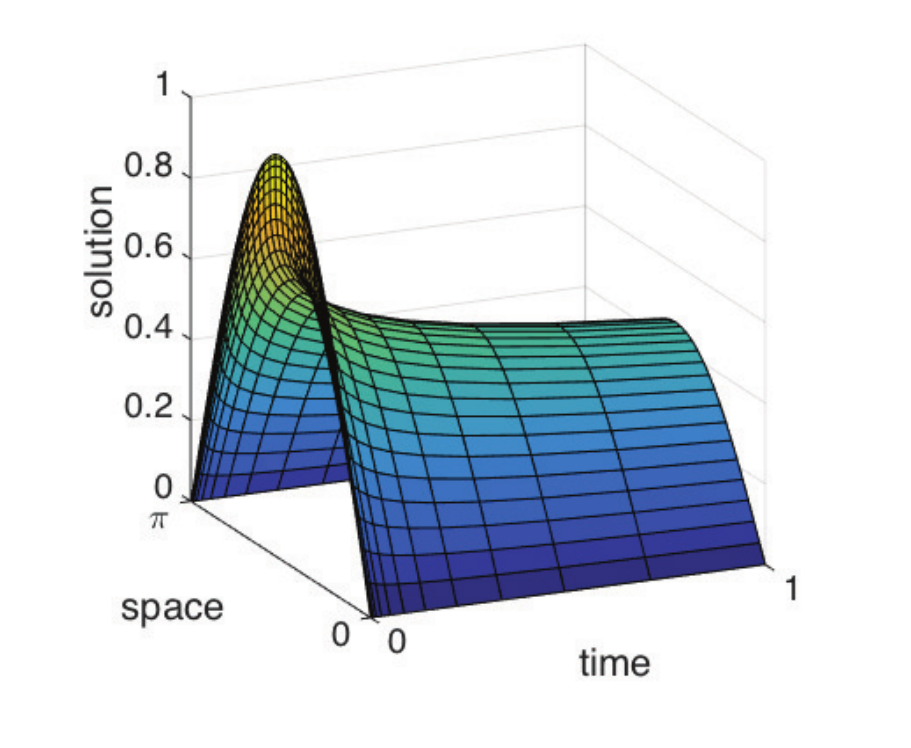}
& 
\includegraphics[width = 0.4\textwidth]{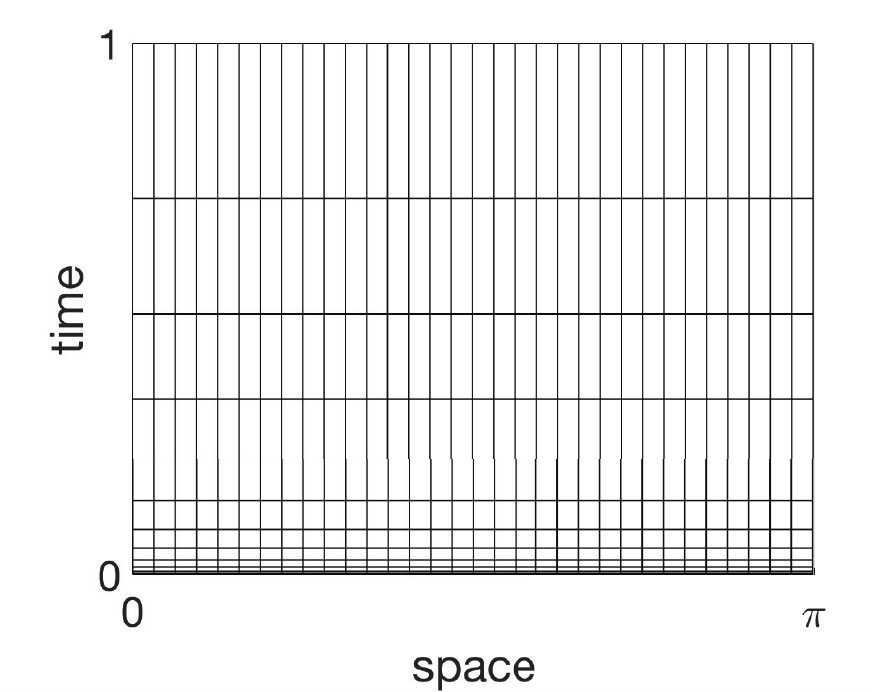}
\\
(a) & (b) 
\end{tabular}
\caption{(a) Analytical solution $u(x,t)$ of the first test problem, for fractional order $\delta=0.6$ and (b) corresponding graded mesh.}
\label{solution_picture}
\end{center}
\end{figure}

As stated in \cite{stynes_graded}, for ``typical'' solutions of~\eqref{model_IVP_1}-\eqref{model_IVP_3}, a rate of convergence of ${\mathcal O}(h^2+\tau^{\delta})$ is obtained when the discrete scheme on uniform grids is considered, whereas the convergence order improves to ${\mathcal O}(h^2+M^{-(2-\delta)})$ with the use of graded meshes. This can be seen in Figure \ref{error_reduction}, where for a fractional order $\delta = 0.4$ and both uniform and graded meshes, we display the maximum errors between the analytical and the numerical solution for various numbers of time-steps $M$ and by using a sufficiently fine spatial grid ($N=2048$). It is observed that the convergence order with the uniform grid is $0.4$, whereas it increases to $1.6$ when the graded mesh is used. 
\begin{figure}[htb]
\begin{center}
\includegraphics[width = 0.5\textwidth]{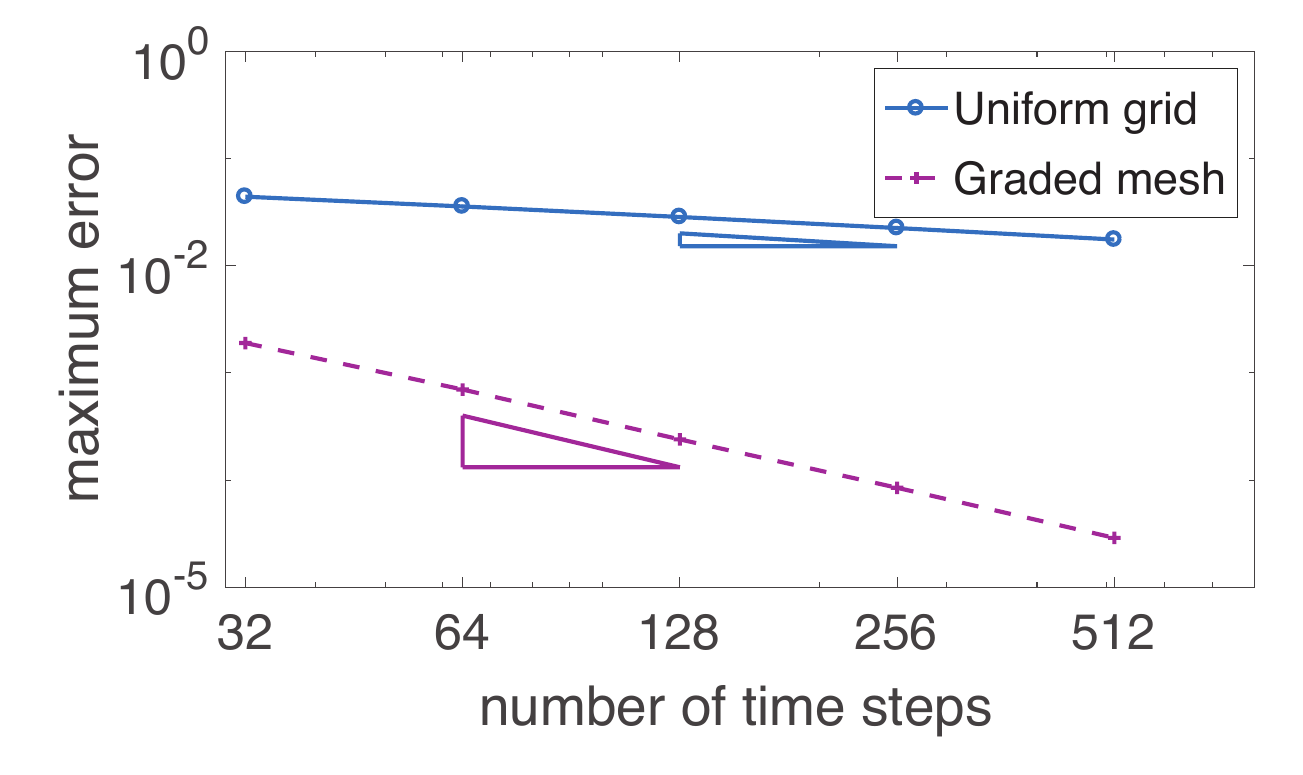}
\caption{Reduction of the errors by using a uniform grid (solid line) and a graded mesh (dotted line) for a fractional order $\delta = 0.4$, and by using a sufficiently fine spatial grid ($N=2048$).}
\label{error_reduction}
\end{center}
\end{figure}

For different values of $\delta$, in Table \ref{table_errors_1d} we display the maximum errors $E_M = \|u-u_{N,M}\|$ obtained when using $N = 1024$ and different values of $M$, together with the corresponding reduction orders computed by $\log_2(E_M/E_{2M})$.
It can be seen that the reduction orders asymptotically match with the expected convergence rates.
\begin{table}[htb]
\begin{center}
\begin{tabular}{|c|c|c|c|c|c|c|}
\cline{1-7}
$\delta$ & & $M=32$ & $M=64$ & $M=128$ & $M=256$ & $M=512$  \\
\hline
%\multirow{2}{*}{$\delta = 0.4$} & $E_M$ & $1.9\times 10^{-3}$ & $7.0\times 10^{-4}$ & $2.4\times 10^{-4}$ & $8.5\times 10^{-5}$ & $2.9\times 10^{-5}$ \\
\multirow{2}{*}{$0.4$} & $E_M= \|u-u_{N,M}\|$ & $1.9E{-3}$ & $7.0E{-4}$ & $2.4E{-4}$ & $8.5E{-5}$ & $2.9E{-5}$ \\
& $\log_2(E_M/E_{2M})$ & $1.44$ & $1.50$ & $1.53$ & $1.55$ & \\
\hline
%\multirow{2}{*}{$\delta = 0.6$} & $E_M$ & $3.3\times 10^{-3}$ & $1.4\times 10^{-3}$ & $5.5\times 10^{-4}$ & $2.1\times 10^{-4}$ & $8.3\times 10^{-5}$\\ 
\multirow{2}{*}{$0.6$} & $E_M= \|u-u_{N,M}\|$ & $3.3E{-3}$ & $1.4E{-3}$ & $5.5E{-4}$ & $2.1E{-4}$ & $8.3E{-5}$\\ 
& $\log_2(E_M/E_{2M})$ & $1.23$ & $1.35$ & $1.35$ & $1.36$ & \\
\hline
%\multirow{2}{*}{$\delta = 0.8$} & $E_M$ & $5.0\times 10^{-3}$ & $2.4\times 10^{-3}$ & $1.1\times 10^{-3}$ & $5.0\times 10^{-4}$ & $2.2\times 10^{-4}$ \\
\multirow{2}{*}{$0.8$} & $E_M= \|u-u_{N,M}\|$ & $5.0E{-3}$ & $2.4E{-3}$ & $1.1E{-3}$ & $5.0E{-4}$ & $2.2E{-4}$ \\
& $\log_2(E_M/E_{2M})$ & $1.05$ & $1.12$ & $1.13$ & $1.14$ & \\
\hline
\end{tabular}
\caption{Maximum errors $E_M = \|u-u_{N,M}\|$ obtained for three different values of $\delta$, with $N = 1024$ spatial grid-points and different numbers of time steps $M$, and corresponding reduction orders $\log_2(E_M/E_{2M})$.}
\label{table_errors_1d}
\end{center}
\end{table}

Next, we study the convergence of the proposed multigrid waveform relaxation method on graded meshes. For this purpose, we perform a semi-algebraic mode analysis which provides very accurate predictions of the asymptotic convergence factors of the multigrid method. In Table \ref{table_SAMA_comparison}, we show the two-grid convergence factors with one smoothing step provided by the analysis for four values of the fractional order $\delta$ and different values of $N$ and $M$.  The convergence factors obtained by using a multilevel $W$-cycle in numerical experiments are displayed (in brackets) in this table as well, showing a good agreement between the predictions and the real convergence rates.  This demonstrates that the analysis is a very useful tool for studying the convergence of the proposed multigrid waveform relaxation method. 
\begin{table}[htb]
\begin{center}
\begin{tabular}{cccc}
\hline
$\delta $ & $64\times 64$ & $128\times 128$ & $256 \times 256$ \\
\hline 
$0.2$ & 0.118 (0.118) & 0.114 (0.116) & 0.114 (0.114)  \\
$0.4$ & 0.123 (0.124) & 0.122 (0.123) & 0.115 (0.117) \\
$0.6$ & 0.106 (0.107) & 0.088 (0.089) & 0.067 (0.068) \\
$0.8$ & 0.066 (0.067) & 0.043 (0.045) & 0.026 (0.028)\\
\hline
\end{tabular}
\caption{Comparison between SAMA predictions and experimentally computed convergence factors (in brackets) with one smoothing step, for four values of $\delta$ and different space-time grid sizes. }
\label{table_SAMA_comparison}
\end{center}
\end{table}

Since we look for an efficient solver that is robust with respect to the number of time steps ($M$), in Figure \ref{LFA_variable_M}, we show the two-grid convergence factors predicted by the analysis for different values of $M$ as well as different fractional orders $\delta$, for a fixed value of $N = 128$.  Indeed, the convergence rates of the proposed multigrid method are bounded below $0.2$ for any values of the parameters $\delta$ and $M$ which demonstrates its robustness.
\begin{figure}[htb]
\begin{center}
\includegraphics[width = 0.8\textwidth]{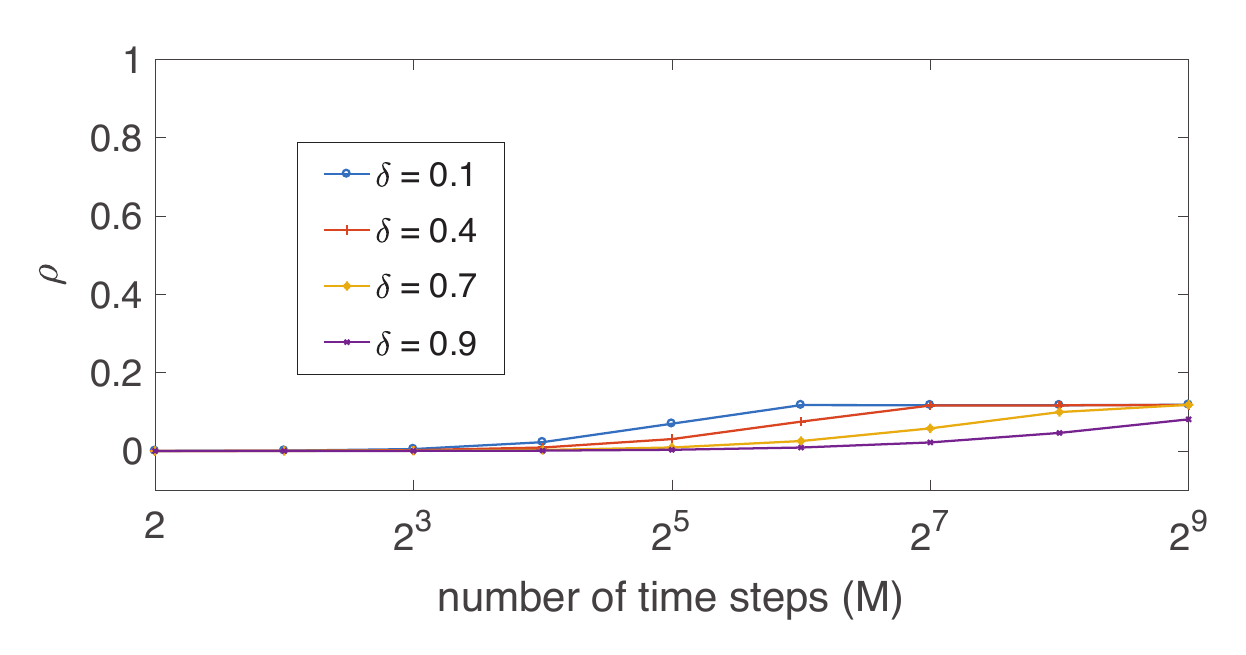}
\caption{Two-grid convergence factors predicted by SAMA for different numbers of time steps, $M$, and four fractional orders $\delta$, for a fixed number of spatial grid-points, $N = 128$.}
\label{LFA_variable_M}
\end{center}
\end{figure}

In order to have a robust multigrid solver, the convergence rate of the method should also be independent of the number of spatial grid-points, $N$.  We perform the semi-algebraic analysis to study the performance for different values of $N$, and a fixed number of time-steps $M=256$. The results are presented in Figure \ref{LFA_variable_N}, where the two-grid convergence factors predicted by SAMA are shown for different values of $\delta$ and for a wide range of values of $N=2^l$ with $l=1,\ldots, 20$.
\begin{figure}[htb]
\begin{center}
\includegraphics[width = 0.75\textwidth]{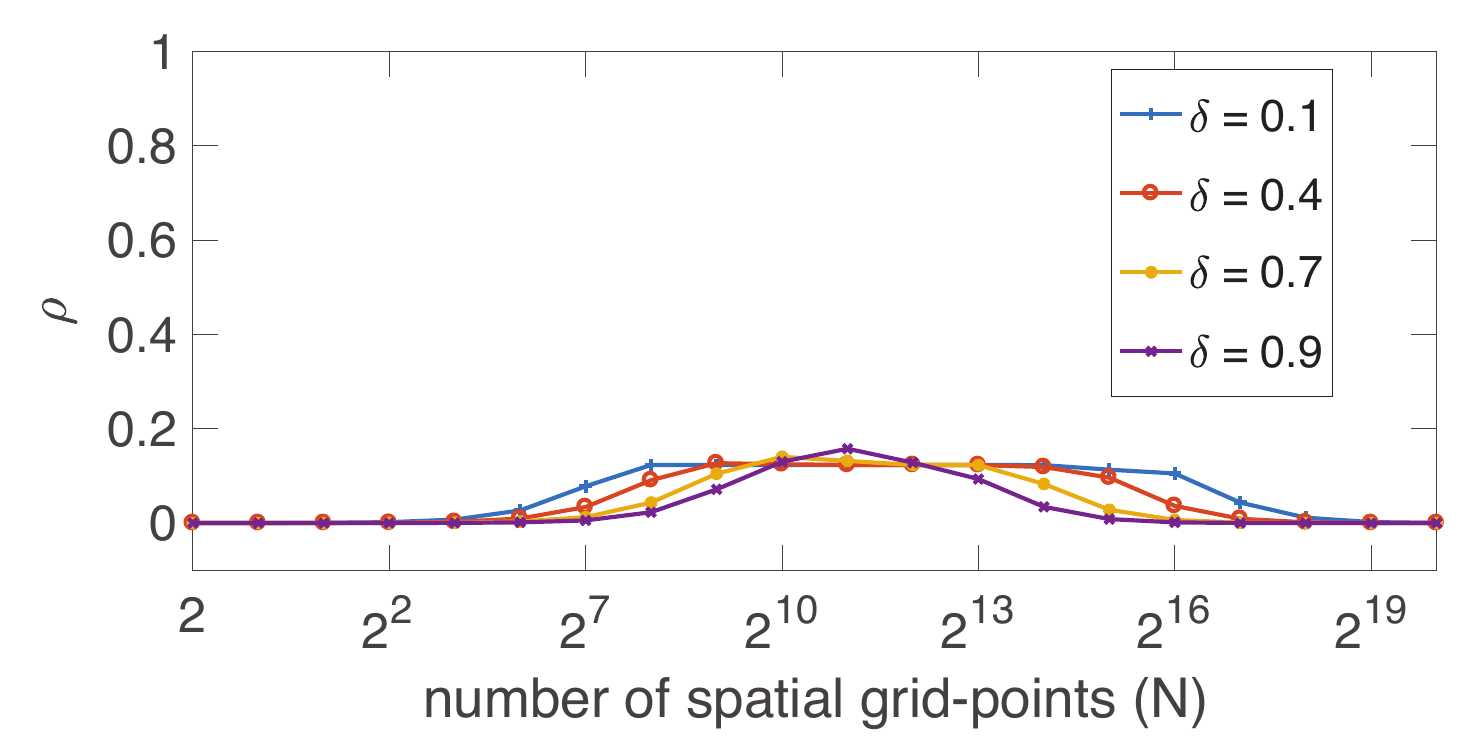}
\caption{Two-grid convergence factors predicted by SAMA for different numbers of spatial grid-points $N$ and four fractional orders $\delta$, for a fixed number of time-steps $M=256$.}
\label{LFA_variable_N}
\end{center}
\end{figure}
We can observe that in all cases the obtained convergence factors are bounded below $0.2$, providing a robust solver for the considered time-fractional problem for any value of the fractional order $\delta$.

In practice, $V$-cycle multigrid schemes are often preferred comparing with $W$-cycles because of their lower computational cost. Here we use $V$-cycles in the following numerical experiments since, as we see, they provide a convergence that is not far from that of the two-grid or $W$-cycles. 

We consider a $V$-cycle with no pre-smoothing and only one post-smoothing step. We choose a $V(0,1)$ instead of a $V(1,0)-$cycle since the first approach provides better convergence factors.
In Table \ref{table_it_1d}, we display the number of WRMG iterations necessary to reduce the initial residual in a factor of $10^{-10}$ for different values of the fractional order $\delta$ and different grid-sizes varying from $128\times 128$ to $2048\times 2048$ doubling the mesh-size in both spatial and temporal dimensions.  We also show the obtained average convergence factors. We can conclude that the convergence of the considered WRMG is very robust independently of the spatial discretization parameter and of the use of the graded mesh.\\

\begin{table}[htb]
\begin{center}
\begin{tabular}{cccccc}
\hline
$\delta$ & $128\times 128$ & $256\times 256$ & $512\times 512$ & $1024\times 1024$ & $2048\times 2048$ \\
\hline
%------------------------------Paco's laptop--------------------------
%0.2 & 11 (0.11) 5.17s & 11 (0.11) 22.80s & 11 (0.11) 97.24s &  11 (0.11)  441.25s &  (0.) s \\
%0.4 & 11 (0.11) 5.23s & 11 (0.11) 22.50s & 11 (0.11) 101.45s & 10 (0.09) 374.85s &  (0.) s \\
%0.6 & 10 (0.09) 4.79s & 9  (0.06) 18.40s & 8   (0.05) 70.64s &   7 (0.04)  284.80s &  (0.) s \\
%0.8 & 8   (0.04) 3.81s & 7  (0.04) 14.08s & 7   (0.04) 71.07s &   7 (0.04)  305.10s & 7 (0.04) 1420.4s \\
%-------------------------------my laptop------------------------------
%0.2 & 11 (0.11) 1.86s & 11 (0.11) 7.71s & 11 (0.11) 36.06s &  11 (0.11)  176.17s & 11 (0.11) 685.09s \\
%0.4 & 11 (0.11) 1.81s & 11 (0.11) 7.67s & 11 (0.11) 34.46s & 10 (0.09)  149.59s &  9 (0.08) 598.86s \\
%0.6 & 10 (0.09) 1.59s & 9  (0.06) 6.38s & 8   (0.05) 23.01s &   7 (0.04)  118.09s &  7 (0.04) 502.67s \\
%0.8 & 8   (0.04) 1.39s & 7  (0.04) 4.49s & 7   (0.04) 23.56s &   7 (0.04)  118.63s &  7 (0.04) 553.54s \\
%------------------------------sin cpu-----------------------------------
0.2 & 11 (0.11) & 11 (0.11) & 11 (0.11) &  11 (0.11)  & 11 (0.11) \\
0.4 & 11 (0.11) & 11 (0.11) & 11 (0.11) & 10 (0.09)  &  9 (0.08) \\
0.6 & 10 (0.09) & 9  (0.06) & 8   (0.05) &   7 (0.04) &  7 (0.04) \\
0.8 & 8   (0.04) & 7  (0.04) & 7   (0.04) &   7 (0.04) &  7 (0.04) \\
\hline
\end{tabular}
\caption{Number of $V(0,1)-$WRMG iterations necessary to reduce the initial residual in a factor of $10^{-10}$ for different fractional orders $\delta$ and for different grid-sizes. The corresponding average convergence factors (in brackets) are also included.}
\label{table_it_1d}
\end{center}
\end{table}

Next, we investigate the effect of the rank $k$ on the convergence behavior of the V-cycle multigrid. Here, we fix the grid-size to be $512 \times 512$ and vary the fractional order $\delta$ and rank $k$. The results are shown in Table~\ref{table_it_k_1d}. As previously, we display the number of WRMG iterations necessary to reduce the initial residual in a factor of $10^{-10}$, together with the corresponding average convergence factors. It is clear that $k$ does not affect the performance of the V-cycle multigrid, more precisely, the number of iterations stays the same and the convergence factors also remain constant. This is expected because the rank $k$ only affects the approximation of the $\mathcal{H}$-matrix representation.  %However, since WRMG is directly applied based on the $\mathcal{H}$-matrix representation, it should not be affected by the rank $k$.

\begin{table}[htb]
\begin{center}
\begin{tabular}{ccccccc}
\hline
$\delta$ & $k=5$ & $k=10$ & $k=15$ & $k=20$ & $k=25$ & $k=30$ \\
\hline
%------------------------------Xiaozhe's laptop-----------------------------------
0.2 & 11 (0.11) & 11 (0.11) & 11 (0.11) & 11 (0.11)  &  11 (0.11) & 11 (0.11)\\
0.4 & 11 (0.11) & 11 (0.11) & 11 (0.11) &  11 (0.11)  & 10 (0.11) & 11 (0.11) \\
0.6 & 8 (0.05) & 8 (0.05) & 8 (0.05) &   8 (0.05) &  8 (0.05) & 8 (0.05) \\
0.8 & 7  (0.04) & 7 (0.04) & 7   (0.04) &   7 (0.04) &  7 (0.04) & 7 (0.04)\\
\hline
\end{tabular}
\caption{
Number of $V(0,1)-$WRMG iterations necessary to reduce the initial residual in a factor of $10^{-10}$ for different fractional orders $\delta$ and for different values of $k$ (fixed grid-size $512\times512$). The corresponding average convergence factors (in brackets) are also included.
}
\label{table_it_k_1d}
\end{center}
\end{table}

Taking into account the excellent convergence rates obtained and that, as we previously commented, the total computational cost is ${\cal O}(kNM\log(M))$ thanks to the use of the ${\cal H}$-matrix representation, we provide a very efficient solver for the time-fractional heat equation on graded meshes.
In Figure \ref{cputime_1d}, we show the CPU time of the proposed WRMG method for different fractional orders. We observe that, for all cases, the computational complexity is optimal, which confirms our expectation.

\begin{figure}[htb]
\centering\includegraphics[scale = 0.35]{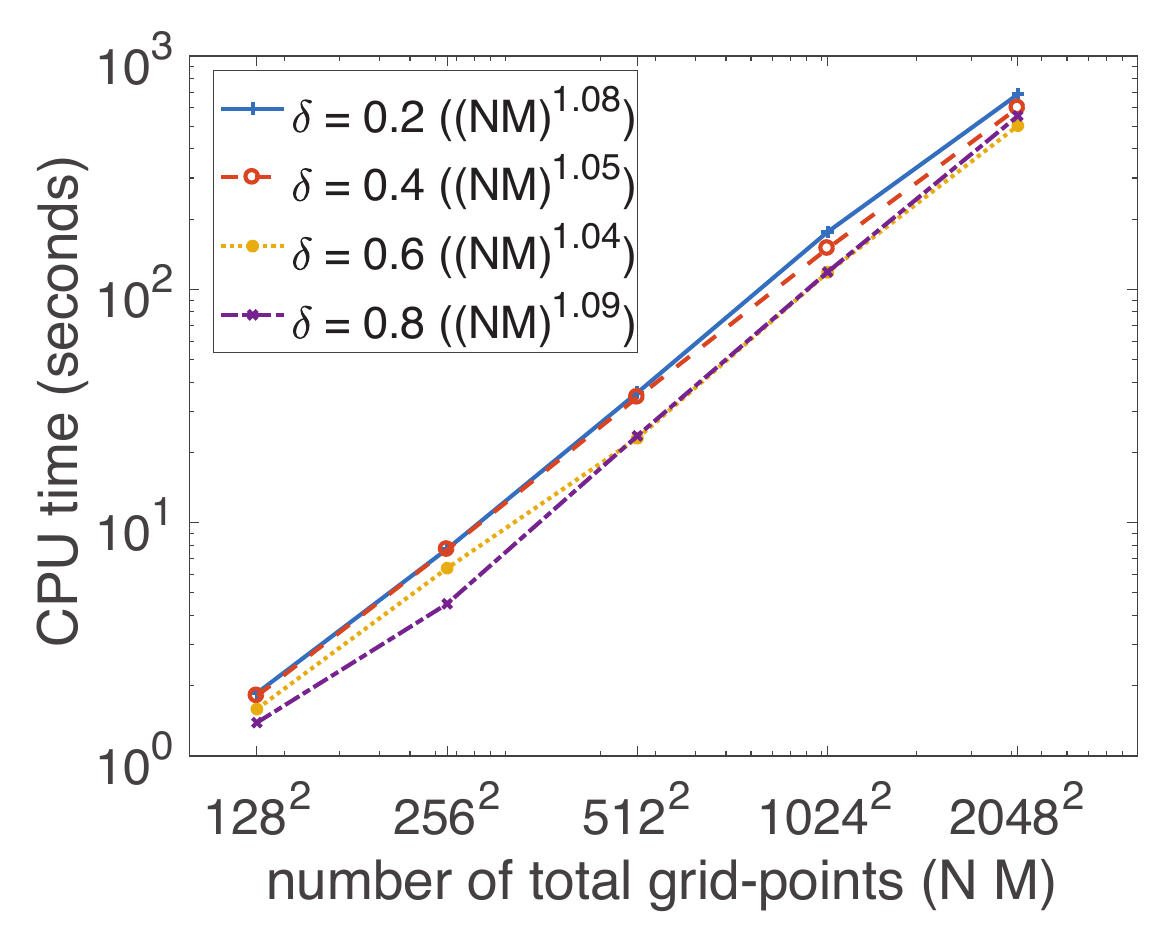}
\caption{CPU time (in seconds) of WRMG method using ${\cal H}$-matrices on nonuniform grids for the first numerical experiment and for different fractional indexes.}
\label{cputime_1d}
\end{figure}

%\noindent{\bf Two-dimensional time-fractional heat equation.}
\subsection{Two-dimensional time-fractional heat equation}
The aim of this second numerical experiment is to show that the proposed strategy can be extended to problems with higher spatial dimensions.  In particular, here we consider a two-dimensional time-fractional diffusion model problem defined on the spatial domain $\Omega=(0,\pi)\times (0,\pi)$ given by
\begin{eqnarray}
D_t^{\delta} u - \Delta u &=& f(x,y,t),\quad (x,y)\in\Omega, \; t>0, \label{2D_model_BVP_1_ex}\\
u(x,y,t) &=& 0,\quad (x,y)\in \partial\Omega, \; t>0,\label{2D_model_BVP_2_ex}\\
u(x,y,0) &=& 0, \quad (x,y)\in\overline{\Omega},\label{2D_model_BVP_3_ex}
\end{eqnarray}
where the right-hand side $f$ is defined as
$$f(x,y,t) = 2(t^3+t^{\delta})\sin x  \sin y +
       \left(\Gamma(\delta+1)+\frac{\Gamma(4)}{\Gamma(4-\delta)}t^{3-\delta}\right)
        \sin x \sin y.$$  
It can be easily seen that the analytic solution is $$u(x,y,t) = (t^3+t^{\delta}) \sin\, x \sin \,y.$$

A semi-algebraic analysis is performed analogously to the one-dimensional spatial case. The only difference is that a standard two-dimensional local Fourier analysis is used now in the spatial domain, combined again with an exact analytical approach in time.  First of all, in Figure \ref{SAMA_2D_compare}, the asymptotic convergence factors obtained by using a multilevel $W$-cycle are compared to the two-grid convergence rates predicted by the semi-algebraic analysis.  This comparison is done by choosing $\delta = 0.4$, $M=128$ time-steps, and for a range of values of $N$ from $N=8 \times 8$ to $N=512 \times 512$.  Accurate correspondence can be observed between the real and the predicted values and such comparisons are similar for different values of the fractional order, which shows that the analysis provides a very useful tool for the study of the convergence of the proposed method. 
\begin{figure}[htb]
\begin{center}
\includegraphics[width = 0.6\textwidth]{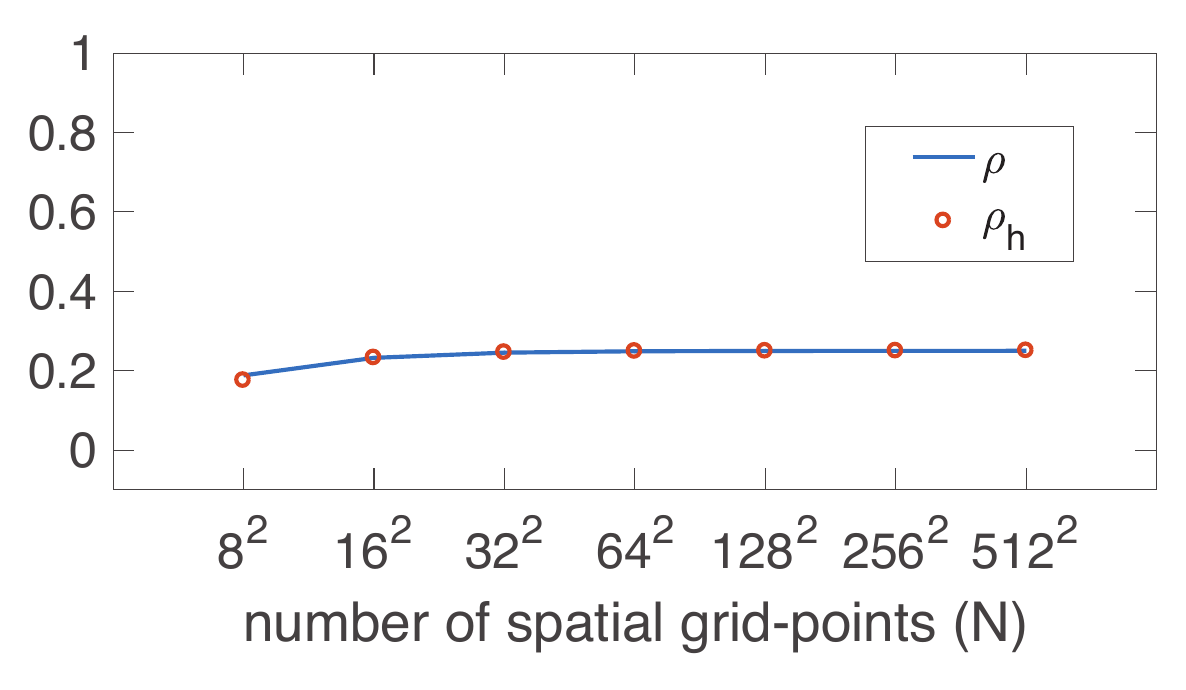}
\caption{Comparison between the two-grid convergence factors predicted by the analysis ($\rho$) and the asymptotic convergence factor of a $W(1,0)$-cycle experimentally computed ($\rho_h$), for different numbers of spatial grid-points ($N$), fractional order $\delta = 0.4$, and for a fixed number of time-steps $M=128$.}
\label{SAMA_2D_compare}
\end{center}
\end{figure}

The semi-algebraic analysis is used now to demonstrate the robustness of the multigrid waveform relaxation method with respect to the fractional order $\delta$.  To this end, in Figure \ref{LFA_variable_N_2D}, the two-grid convergence factors provided by SAMA are shown for different values of $\delta$ and different numbers of spatial grid-points $N$, for a fixed number of time-steps $M=128$.  A very satisfactory convergence is observed in all cases, making the multigrid waveform relaxation method a good choice for an efficient solution of the time-fractional two-dimensional heat equation. 
\begin{figure}[htb]
\begin{center}
\includegraphics[width = 0.75\textwidth]{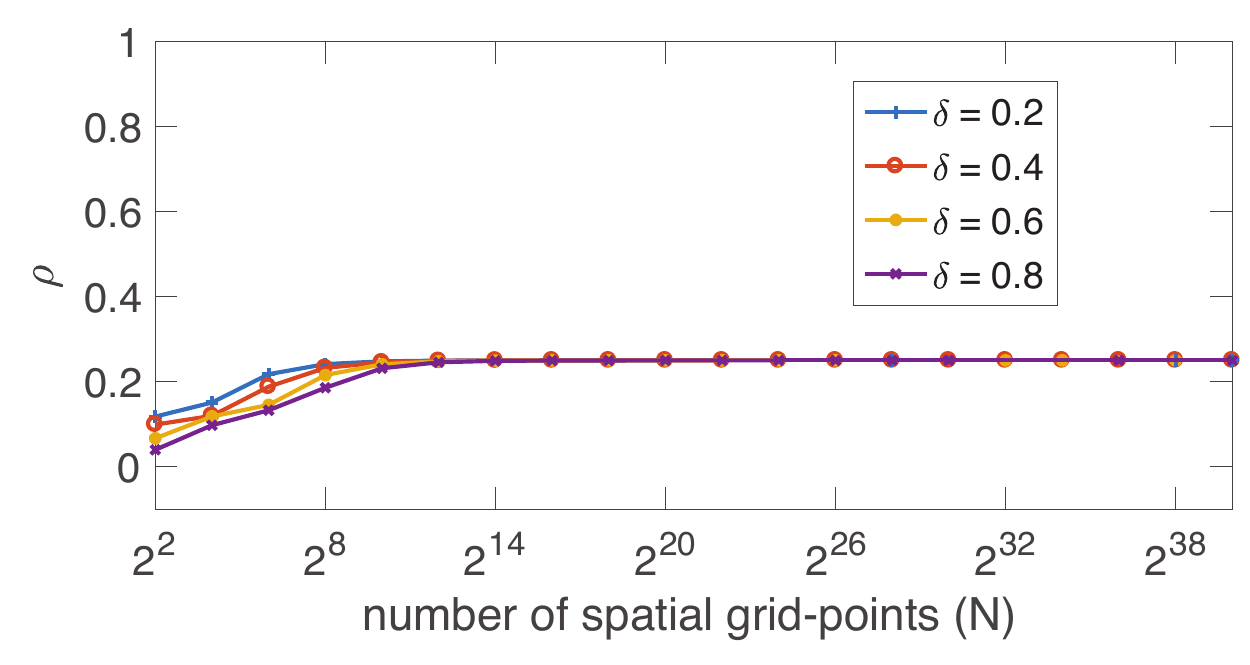}
\caption{Two-grid convergence factors predicted by SAMA for different numbers of spatial grid-points ($N$) and four fractional orders $\delta$, when considering $M=128$ time-steps.}
\label{LFA_variable_N_2D}
\end{center}
\end{figure}

In order to show the efficiency and the robustness of the proposed method in the case of two spatial dimensions, we consider a multigrid $V$-cycle with one pre- and one post-smoothing steps. As for the one-dimensional case, in Table~\ref{table_it_2d}, we show a performance of the method independent of the grid-size and robust with respect to the value of the fractional order $\delta$. In particular, we show the number of WRMG iterations that are needed to reduce the initial maximum residual by a factor of $10^{-10}$ for different grid-sizes varying from $32\times 32\times 32$ to $256\times 256\times 256$ and for different values of the fractional order $\delta$. The average convergence factors are displayed as well.
\begin{table}[htb]
\begin{center}
\begin{tabular}{ccccc}
\hline
$\delta$ & $32\times 32\times 32$ & $64\times 64\times 64$ & $128\times 128\times 128$ & $256\times 256\times 256$  \\
\hline
%0.2 & 11 (0.09) 2.41s & 11 (0.10) 23.84s & 12 (0.11) 280.59s & 12 (0.11) 2767.20s \\
%0.4 & 11 (0.09) 2.36s & 11 (0.10) 24.65s & 11 (0.10) 257.51s & 12 (0.11) 2679.265s \\
%0.6 & 10 (0.08) 2.07s & 11 (0.10) 25.59s & 11 (0.10) 334.74s & 11 (0.10) 2609.01s \\
%0.8 & 10 (0.08) 2.15s & 11 (0.10) 24.90s & 11 (0.10) 304.65s & 11 (0.10) 2270.60s \\
%------------------------------sin cpu time------------------------------------------------------------------
0.2 & 9 (0.07) & 9 (0.07) & 9 (0.07) & 9 (0.07)  \\
0.4 & 9 (0.07) & 9 (0.08) & 9 (0.08) & 9 (0.08)  \\
0.6 & 9 (0.07) & 9 (0.08) & 9 (0.08) & 9 (0.08)  \\
0.8 & 9 (0.07) & 9 (0.08) & 9 (0.08) & 9 (0.08)  \\
\hline
\end{tabular}
\caption{Number of $V(1,1)$-WRMG iterations necessary to reduce the initial residual in a factor of $10^{-10}$, together with the corresponding average convergence factors (in brackets), for different fractional orders $\delta$ and for different grid-sizes.}
\label{table_it_2d}
\end{center}
\end{table}

In Table~\ref{table_it_k_2d}, we study the influence of the rank $k$ in the WRMG performance. We fix the grid-size to be $64 \times 64 \times 64$ and vary the fractional order $\delta$ and the rank $k$. We can see from Table~\ref{table_it_k_2d} that $k$ does not affect the performance of the V-cycle multigrid, since the number of iterations stays the same, as well as the convergence factors remain constant. This demonstrates the robustness of our $\mathcal{H}$-matrix representation. 

\begin{table}[htb]
\begin{center}
\begin{tabular}{ccccccc}
\hline
$\delta$ & $k=5$ & $k=10$ & $k=15$ & $k=20$  & $k=25$ & $k=30$ \\
\hline
0.2 & 9 (0.07) & 9 (0.07) & 9 (0.07) & 9 (0.07)  & 9 (0.07)  & 9 (0.07)   \\
0.4 & 9 (0.07) & 9 (0.08) & 9 (0.07) & 9 (0.08) & 9 (0.08) & 9 (0.07)  \\
0.6 & 9 (0.08) & 9 (0.08) & 9 (0.08) & 9 (0.08) & 9 (0.08) & 9 (0.08)  \\
0.8 & 9 (0.08) & 9 (0.08) & 9 (0.08)&  9 (0.08) & 9 (0.08)& 9 (0.08) \\
\hline
\end{tabular}
\caption{
Number of $V(1,1)$-WRMG iterations necessary to reduce the initial residual in a factor of $10^{-10}$, together with the corresponding average convergence factors (in brackets), for different fractional orders $\delta$ and for different rank $k$ (fixed grid-size $64\times64\times64$).  
}
\label{table_it_k_2d}
\end{center}
\end{table}

Again, we obtain an excellent convergence that together with the computational cost of ${\cal O}(kNM\log(M))$, makes the proposed WRMG method a very efficient solver also for the two-dimensional time-fractional heat equation on graded meshes. 
The CPU time of the proposed WRMG method is shown in Figure \ref{cputime_2d} for different fractional orders, where we can confirm that the computational complexity is optimal for all cases, as expected.
\begin{figure}[htb]
\centering\includegraphics[scale = 0.35]{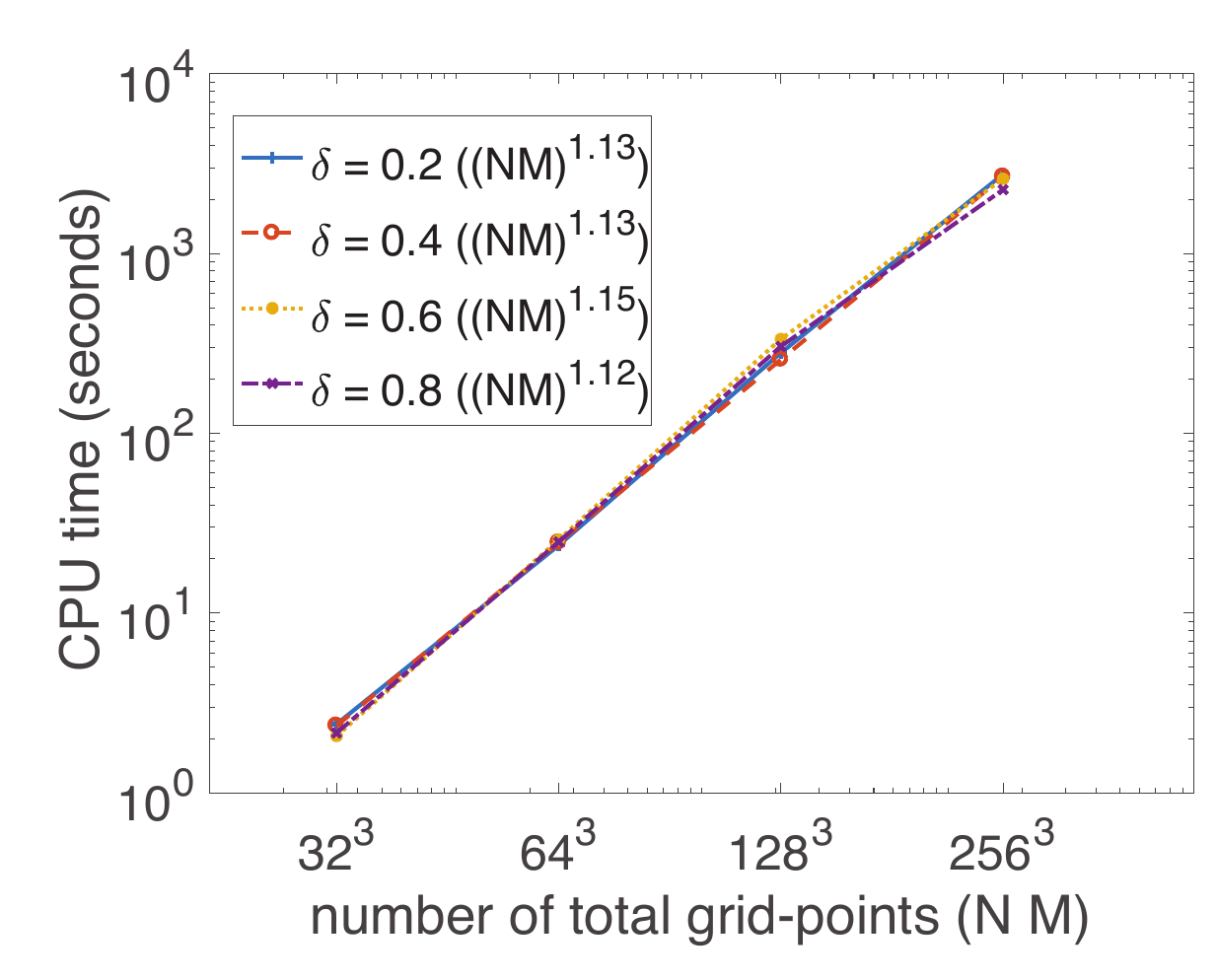}
\caption{CPU time (in seconds) of WRMG method using ${\cal H}$-matrices on nonuniform grids for the second numerical experiment and for different fractional indexes.}
\label{cputime_2d}
\end{figure}

%%%%%%%%%%%%%%%%%%%%%%%%%%%%%%%%%%%%%%%%%%%%%%%%%

\section{Conclusions}\label{sec:6}

In this work, we have proposed a fast solver for the time-fractional heat equation, targeting in particular ``typical'' non-smooth solutions that arise even when right-hand sides are smooth.  The use of uniform grids in these cases may lead to a very poor convergence of the numerical solution, which can be enhanced by considering graded meshes. 
The algorithm is based on the multigrid acceleration of the waveform relaxation method, which provides an optimal complexity only for uniform grids.  For non-uniform grids in time, however, the computational cost increases and is not optimal any more.  In order to overcome this drawback and keep the optimality of the overall computational complexity of the method, the hierarchical matrices framework is considered.  In this way, a computational cost of ${\cal O}(kNM\log(M))$, where $M$ is the number of time steps and $N$ is the number of spatial grid points, is obtained.  Numerical experiments with one- and two-spatial dimensions are presented to illustrate the efficiency of the method and its robustness with respect to the fractional orders. Within these tests, a semi-algebraic mode analysis is used to theoretically justify the good convergence rates provided by the multigrid waveform relaxation algorithm. 

One direction for future work is to study how to generalize our fast solver to high order schemes. This involves rewriting the high order schemes using the finite element framework. If this is doable, we can replace the original kernel by a separable kernel and construct the corresponding $\mathcal{H}$-matrix representations.  Otherwise, our method cannot be directly applied and we need to investigate other ways to construct the low rank approximations, for example, construct them algebraically. Another possible research direction is the generalization of the proposed approach to time-fractional nonlinear problems.  We plan to use the $\mathcal{H}$-matrix representation to approximate the time-fractional derivatives and nonlinear multigrid schemes to handle the nonlinearity, for example,  using the full approximation scheme~\cite{Bra77}.

%%%%%%%%%%%%%%%%%%%%%%%%%%%%%%%%%%%%%%%%%%%%%%%%%

\section*{Acknowledgements}

Francisco J. Gaspar has received funding from the European Union's Horizon 2020 research and innovation programme under the Marie Sklodowska-Curie grant agreement No 705402, POROSOS. The work of Carmen Rodrigo is supported in part by the Spanish project FEDER /MCYT MTM2016-75139-R and the Diputaci\'on General de Arag\'on (Grupo de Referencia APEDIF, ref. E24\_17R). The work of Hu is partially supported by NSF grant DMS-1620063.

%\section*{References}

\bibliography{references}

\begin{thebibliography}{10}
\expandafter\ifx\csname url\endcsname\relax
  \def\url#1{\texttt{#1}}\fi
\expandafter\ifx\csname urlprefix\endcsname\relax\def\urlprefix{URL }\fi
\expandafter\ifx\csname href\endcsname\relax
  \def\href#1#2{#2} \def\path#1{#1}\fi

\bibitem{cushman}
J.~H. Cushman, T.~R. Ginn, Nonlocal dispersion in media with continuously
  evolving scales of heterogeneity, Transport in Porous Media 13~(1) (1993)
  123--138.

\bibitem{hall}
M.~G. Hall, T.~R. Barrick, From diffusion-weighted {MRI} to anomalous diffusion
  imaging, Magnetic Resonance in Medicine 59~(3) (2008) 447--455.

\bibitem{hilfer2}
R.~Hilfer, Applications of Fractional Calculus in Physics, World Scientific,
  Singapore, 2000.

\bibitem{lazarov}
B.~Jin, R.~Lazarov, J.~Pasciak, Z.~Zhou, Error analysis of semidiscrete finite
  element methods for inhomogeneous time-fractional diffusion, IMA Journal of
  Numerical Analysis 35 (2015) 561--582.

\bibitem{li_xu}
X.~Li, C.~Xu, Existence and uniqueness of the weak solution of the space-time
  fractional diffusion equation and a spectral method approximation,
  Communications in Computational Physics 8 (2010) 1016--1051.

\bibitem{11hesthaven}
J.~T. Machado, V.~Kiryakova, F.~Mainardi, Recent history of fractional
  calculus, Communications in Nonlinear Science and Numerical Simulation 16~(3)
  (2011) 1140 -- 1153.

\bibitem{metzler_klafter}
R.~{Metzler}, J.~{Klafter}, The random walk's guide to anomalous diffusion: a
  fractional dynamics approach, Physics Reports 339 (2000) 1--77.

\bibitem{purohit}
S.~D. Purohit, Solutions of fractional partial differential equations of
  quantum mechanics, Advances in Applied Mathematics and Mechanics 5 (2013)
  639--651.

\bibitem{Ren.J;Sun.Z;Zhao.X2013a}
J.~Ren, Z.-Z. Sun, X.~Zhao, Compact difference scheme for the fractional
  sub-diffusion equation with neumann boundary conditions, Journal of
  Computational Physics 232~(1) (2013) 456--467.

\bibitem{10hesthaven}
J.~Tenreiro~Machado, A.~M. Galhano, J.~J. Trujillo, Science metrics on
  fractional calculus development since 1966, Fractional Calculus and Applied
  Analysis 16~(2) (2013) 479--500.

\bibitem{Zhao.X;Sun.Z2011a}
X.~Zhao, Z.-z. Sun, A box-type scheme for fractional sub-diffusion equation
  with neumann boundary conditions, Journal of Computational Physics 230~(15)
  (2011) 6061--6074.

\bibitem{Zhang2014}
Y.-N. Zhang, Z.-Z. Sun, Error analysis of a compact adi scheme for the 2d
  fractional subdiffusion equation, J. Sci. Comput. 59~(1) (2014) 104--128.

\bibitem{Lin2016}
X.-L. Lin, X.~Lu, M.~K. Ng, H.-W. Sun, A fast accurate approximation method
  with multigrid solver for two-dimensional fractional sub-diffusion equation,
  Journal of Computational Physics 323 (2016) 204--218.

\bibitem{LionsMaday}
J.~Lions, Y.~Maday, G.~Turinici, A ``parareal'' in time discretization of
  {PDE}'s, C. R. Acad. Sci. S\'er. I Math. 332 (2001) 661 -- 668.

\bibitem{Xu_Hesthaven_Chen}
Q.~Xu, J.~S. Hesthaven, F.~Chen, A parareal method for time-fractional
  differential equations, Journal of Computational Physics 293 (2015) 173--183.

\bibitem{fast_Caputo}
S.~Jiang, J.~Zhang, Q.~Zhang, Z.~Zhang, Fast evaluation of the caputo
  fractional derivative and its applications to fractional diffusion equations,
  Communications in Computational Physics 21.
\newblock \href {http://dx.doi.org/10.4208/cicp.OA-2016-0136}
  {\path{doi:10.4208/cicp.OA-2016-0136}}.

\bibitem{SISC_Gaspar}
F.~J. Gaspar, C.~Rodrigo, Multigrid waveform relaxation for the time-fractional
  heat equation, SIAM Journal on Scientific Computing 39~(4) (2017)
  A1201--A1224.

\bibitem{nonuniform}
H.-l. Liao, W.~Mclean, J.~Zhang, A second-order scheme with nonuniform time
  steps for a linear reaction-sudiffusion problem\href
  {http://dx.doi.org/10.13140/RG.2.2.16993.81766}
  {\path{doi:10.13140/RG.2.2.16993.81766}}.

\bibitem{fractional_xiaozhe}
X.~Zhao, X.~Hu, W.~Cai, G.~Karniadakis, Adaptive finite element method for
  fractional differential equations using hierarchical matrices, Computer
  Methods in Applied Mechanics and Engineering.

\bibitem{Bebendorf.M2008a}
M.~Bebendorf, Hierarchical matrices: A Means to Efficiently Solve Elliptic
  Boundary Value Problems, Springer, 2008.

\bibitem{H_matrices_Hackbusch_book}
W.~Hackbusch, Hierarchical Matrices: Algorithms and Analysis, Vol.~49 of
  Springer Series in Computational Mathematics, Springer-Verlag Berlin
  Heidelberg, 2015.

\bibitem{horton_vandewalle}
G.~Horton, S.~Vandewalle, P.~Worley, An algorithm with polylog parallel
  complexity for solving parabolic partial differential equations, SIAM Journal
  on Scientific Computing 16~(3) (1995) 531--541.

\bibitem{lubich_ostermann}
C.~Lubich, A.~Ostermann, Multigrid dynamic iteration for parabolic equations,
  BIT 27 (1987) 216--234.

\bibitem{Vandewalle_book}
S.~Vandewalle, Parallel Multigrid waveform relaxation for parabolic problems,
  B.G. Teubner Stuttgart, 1993.

\bibitem{Stu_Tro}
K.~St\"uben, U.~Trottenberg, Multigrid methods: Fundamental algorithms, model
  problem analysis and applications, in: W.~Hackbusch, U.~Trottenberg (Eds.),
  Multigrid Methods, Vol. 960 of Lecture Notes in Mathematics, Springer Berlin
  Heidelberg, 1982, pp. 1--176.

\bibitem{TOS01}
U.~Trottenberg, C.~W. Oosterlee, A.~Sch\"uller, Multigrid, Academic Press, New
  York, 2001.

\bibitem{Wess}
P.~Wesseling, An Introduction to Multigrid Methods, John Wiley, Chichester, UK,
  1992.

\bibitem{sama}
S.~Friedhoff, S.~MacLachlan, A generalized predictive analysis tool for
  multigrid methods, Numerical Linear Algebra with Applications 22~(4) (2015)
  618--647.

\bibitem{Bra77}
A.~Brandt, Multi-level adaptive solutions to boundary-value problems,
  Mathematics of Computation 31~(138) (1977) 333--390.

\bibitem{Bra94}
A.~Brandt, Rigorous quantitative analysis of multigrid, {I}: Constant
  coefficients two-level cycle with {L2}-norm, SIAM Journal on Numerical
  Analysis 31~(6) (1994) 1695--1730.

\bibitem{Wie01}
R.~Wienands, W.~Joppich, Practical Fourier analysis for multigrid methods,
  Chapman and Hall/CRC Press, 2005.

\bibitem{Hilfer}
E.~Gerolymatou, I.~Vardoulakis, R.~Hilfer, Modelling infiltration by means of a
  nonlinear fractional diffusion model, Journal of Physics D: Applied Physics
  39~(18) (2006) 4104--4110.

\bibitem{Podlubny}
I.~Podlubny, Fractional Differential Equations: An Introduction to Fractional
  Derivatives, Fractional Differential Equations, to Methods of Their Solution
  and Some of Their Applications, Mathematics in science and engineering,
  Academic Press, 1999.

\bibitem{Diethelm}
K.~Diethelm, The analysis of fractional differential equations. An
  Application-Oriented Exposition Using Differential Operators of Caputo Type,
  volume 2004 of Lecture Notes in Mathematics, Springer-Verlag, Berlin, 2010.

\bibitem{stynes_graded}
M.~Stynes, E.~O'Riordan, J.~L. Gracia, Error analysis of a finite difference
  method on graded meshes for a time-fractional diffusion equation, SIAM
  Journal on Numerical Analysis 55~(2) (2017) 1057--1079.

\bibitem{gamma}
P.~J. Davis, Leonhard {E}uler's integral: A historical profile of the {G}amma
  function, The American Mathematical Monthly 66~(10) (1959) 849--869.

\bibitem{lin_xu}
Y.~Lin, C.~Xu, Finite difference/spectral approximations for the
  time-fractional diffusion equation, Journal of Computational Physics 225~(2)
  (2007) 1533--1552.

\bibitem{Oldham_Spanier}
K.~Oldham, J.~Spanier, The Fractional Calculus: Theory and Applications of
  Differentiation and Integration to Arbitrary Order, Dover books on
  mathematics, Dover Publications, 2006.

\bibitem{ZhongSun_Wu}
Z.~zhong Sun, X.~Wu, A fully discrete difference scheme for a diffusion-wave
  system, Applied Numerical Mathematics 56~(2) (2006) 193--209.
\newblock \href {http://dx.doi.org/https://doi.org/10.1016/j.apnum.2005.03.003}
  {\path{doi:https://doi.org/10.1016/j.apnum.2005.03.003}}.

\bibitem{Alikhanov.A2015a}
A.~A. Alikhanov, A new difference scheme for the time fractional diffusion
  equation, Journal of Computational Physics 280 (2015) 424--438.

\bibitem{Cao.J;Xu.C2013a}
J.~Cao, C.~Xu, A high order schema for the numerical solution of the fractional
  ordinary differential equations, Journal of Computational Physics 238 (2013)
  154--168.

\bibitem{Zhao.X;Sun.Z;Karniadakis.G2015a}
X.~Zhao, Z.-z. Sun, G.~E. Karniadakis, Second-order approximations for variable
  order fractional derivatives: algorithms and applications, Journal of
  Computational Physics 293 (2015) 184--200.

\bibitem{H_matrices_Hackbusch}
S.~Borm, L.~Grasedyck, W.~Hackbusch, Introduction to hierarchical matrices with
  applications, Engineering Analysis with Boundary Elements 27~(5) (2003) 405
  -- 422.

\bibitem{Luchko}
Y.~Luchko, Initial-boundary-value problems for the one-dimensional
  time-fractional diffusion equation, Fractional Calculus and Applied Analysis
  15~(1) (2012) 141--160.

\end{thebibliography}

\end{document}